\documentclass[11pt,reqno]{amsart}
\usepackage[T1]{fontenc}
\usepackage{amsfonts}
\usepackage{amssymb}
\usepackage{mathrsfs}
\usepackage[latin1]{inputenc}
\usepackage{amsmath}
\usepackage[english]{babel}

\newtheorem{theorem}{Theorem}[section]
\newtheorem{lemma}[theorem]{Lemma}

\newtheorem{remark}[theorem]{Remark}
\newtheorem{prop}[theorem]{Proposition}
\newtheorem{example}[theorem]{Example}
\newtheorem{corollary}[theorem]{Corollary}
\newtheorem{hypothesis}[theorem]{Hypothesis}
\newtheorem{hypotheses}[theorem]{Hypotheses}

\numberwithin{equation}{section}
\newcommand{\R}{{\mathbb R}}

\newcommand{\C}{{\mathbb C}}
\newcommand{\N}{{\mathbb N}}

\newcommand{\cL}{{\mathcal L}}

\newcommand{\cB}{{\mathcal B}}

\newcommand{\cU}{{\mathcal U}}

\newcommand{\ve}{\varepsilon}

\newcommand{\su}{\subseteq}

\newcommand{\ov}{\overline}

\newcommand{\ar}{\mbox{\rm arg}\,}
\newcommand{\Lb}{L^{\gamma, b}}

\begin{document}

\title{Analyticity for some degenerate evolution equations defined on domains with corners}

\author{Angela A. Albanese, Elisabetta M. Mangino}



\address{ Angela A. Albanese\\
Dipartimento di Matematica e Fisica ``E. De Giorgi''\\
Universit\`a del Salento\\
I-73100 Lecce, Italy}
\email{angela.albanese@unisalento.it}

\address{ Elisabetta M. Mangino\\
Dipartimento di Matematica e Fisica ``E. De Giorgi''\\
Universit\`a del Salento\\
I-73100 Lecce, Italy} \email{elisabetta.mangino@unisalento.it}

\begin{abstract}
We study the analyticity of the semigroups generated by some classes of
degenerate second order differential operators in the space of continuous function on a domain with corners. These semigroups arise from the theory of dynamics of
populations.
\end{abstract}

\maketitle

\section{Introduction}
In this paper we  deal with the class of degenerate second order elliptic differential operators 
\begin{equation}\label{e.operator}
L=\Gamma(x)\sum_{i=1}^d[\gamma_i(x_i)x_i\partial^2_{x_i}+b_i(x)\partial_{x_i}],\quad x\in Q^d=[0,M]^d,
\end{equation}
where $M>0$,   $\Gamma$, $b_i$ and $\gamma_i$, for $i=1,\ldots,d$, are continuous functions on $Q^d$ and on $[0,M]$ respectively and $b=(b_1, \dots, b_d)$ is an inward pointing drift.
The operator \eqref{e.operator} arises in the theory of Fleming--Viot processes, namely  measure--valued processes that can be viewed as diffusion approximations of empirical processes associated with some classes of discrete time Markov chains in population genetics. We refer to \cite{EK,EK1,FV} for more details on the  topic.
Recent applications of Fleming-Viot processes to the study of the volatility-stabilized markets can be found in \cite{So}.
From the analytic point of view, the interest in the operator (\ref{e.operator}) relies on the fact  that it is of degenerate type and its domain presents edges and corners,  hence, the classical techniques for the study of (parabolic) elliptic operators on smooth domains cannot be applied. 
 
 In the one-dimensional case, the study of such type of degenerate (parabolic) elliptic problems on $C([0,1])$ started in the fifties with the papers by Feller  \cite{F1,F2}, where it is pointed out that the behaviour on the boundary of the diffusion process associated with the degenerate operator  constitutes one of its main characteristics.  
 The subsequent work of Cl\'ement and Timmermans \cite{CT} clarified which conditions on the coefficients of the operator  \eqref{e.operator}  guarantee the generation of a $C_0$--semigroup in $C([0,1])$.  The problem of the regularity of the generated semigroup in $C([0,1])$ has been  considered by several authors, \cite{An,CM,BRS,Met, ACM-1}. In particular,  Metafune \cite{Met} established the analyticity of the semigroup under suitable conditions on the coefficients of the operator, obtaining, among other results,   the analyticity of the semigroup generated by $x(1-x)D^2$ on $C([0,1])$, which was a problem left open for a long time. We refer to \cite{CMPR} for a survey on this topic.
 
The latter result was extended to the multidimensional case in \cite{AM-2}, where the authors proved the analyticity of the semigroup generated by   the 
operator
\[ Au(x)=\frac{1}{2}\sum_{i,j=1}^dx_i(\delta_{ij}-x_j)\partial_{x_ix_j}^2u(x)\]
on $C(S^d)$, where $S^d$ is the $d$-dimensional canonical simplex. On this topic we refer  to the papers \cite{ACM-2,ACM-1,AM,AM-2, CR,CC,E,S1,S2,S3,Stannat1} and the references quoted therein (in particular, see the introduction of \cite{AM-2} for a brief survey of the main results on this operator).

  In \cite{CC1} Cerrai and Cl\'ement established Schauder estimates for \eqref{e.operator}  under suitable H\"older continuity hypothesis on its coefficients.
Analogous estimates, but with different tecniques, where established in \cite{BP} (see also \cite{ABP}) for the same operator defined on the orthant $\mathbb{R}^d_+$ and in \cite{EM, EM-1}  for similar operators defined on domains with corners.

The aim of this paper is to  present some results about generation, sectoriality and gradient estimates for the resolvent of a suitable realization of  
(\ref{e.operator}) in $C(Q^d)$. 
To this end, we start with the analysis in the particular case that the functions $b_i$ are costant and $\Gamma=1$, first in the  one-dimensional case and then, via a tensor product argument, in the multi-dimensional setting. Much attention is paid to the costants appearing in the analyticity and gradient estimates, showing their uniformity
for $b_i$ belonging to an interval $[0,B]$. These results strongly rely on estimates proved in \cite{AM-3}.  
As a consequence, we can treat the case of non-costant drift  with a perturbation argument under the assumption that 
there exists $\delta>0$ and $C>0$ such that, for every $i=1,\ldots,d$ and $x,\ x'\in Q^d$ with $x_i<\delta$ and $x_i'=0$, we have
\[
|b_i(x)-b_i(x')|\leq C\sqrt{x_i},
\]

Finally we treat the case that $\Gamma$ is not a costant function, by applying a "freezing coefficients" proof. An important role in this argument will be played by the uniformity of the costants in the resolvent estimates.

As a by-product of the previous results we obtain analogous results for the operator
\[ \Gamma(x)\sum_{i=1}^d[\gamma_i(x_i)x_i(1-x_i)\partial^2_{x_i}+b_i(x)\partial_{x_i}],\quad x\in [0,1]^d.\]
This will be the starting point for a forthcoming paper on the analyticity of Fleming-Viot type operators defined on the canonical simplex.

\subsection{Notation} The function spaces considered in this paper consist of complex--valued functions.

Let $K\su\R^d$ be a compact set. For $n\in\N$ we denote by  $C^n(K)$ the space of all $n$--times continuously differentiable functions $u$ on $K$ such that $\lim_{x\to x_0}D^{\alpha}u(x)$ exists and is finite for all $|\alpha|\leq n$ and  $x_0\in \partial K$. In particular, $C(K)$ denotes the space of all continuous functions $u$ on $K$. The norm on $C(K)$ is the supremum norm and is  denoted by $\|\ \|_\infty$. The norm $\|\ \|_{n,\infty}$ on  $C^n(K)$  is defined by $\|u\|_{n,\infty}=\sum_{|\alpha|\leq n}\|D^\alpha u\|_\infty$.

Moreover, we denote by $C([0,\infty])$ the Banach space of continuous functions on $[0,\infty[$ converging  at infinity,  endowed with the supremum norm $||\cdot||_\infty$.
Analogously, for every $n\in\N$, $C^n([0,\infty])$ stands for the space of functions $u\in C([0,\infty])$ with derivatives up to order $n$ that have finite limits at $\infty$. Finally $C_c^n([0,\infty[)$ denotes the subspace of $C^n([0,\infty[)$ of functions with compact support and $C_0([0,\infty[)$ denotes the space of continuous functions on $[0,\infty[$ vanishing at $\infty$.

For easy reading, in some  cases we will adopt the notation $\|\varphi(x)u\|_\infty$ to still denote $\sup_{x\in K}|\varphi(x)u(x)|$.


For other undefined notation and results on the theory of semigroups we refer to \cite{EN,L,P}. 

In the present paper we will use some results about injective tensor products of Banach spaces. We refer to  \cite{J,K,T,N} for definitions and  basic results  in this topic and for related applications.

\section{Auxiliary Results}

\markboth{A.\,A. Albanese, E.\, M. Mangino}%
{\MakeUppercase{Analyticity for some degenerate evolution equations }}

\subsection{The one--dimensional case}

Let $M, B\in \R$ with $M, B>0$ and let $\gamma\in C([0,M])$ be  a strictly positive function. Set $\gamma_0:=\min_{x\in [0,M]}\gamma(x)>0$. Let  $b\in [0,B]$ and consider 
   the one--dimensional second order differential operator 
\begin{equation}
L^{\gamma, b}u(x)=\gamma(x)x u''(x)+bu'(x), \quad x\in [0,M].
\end{equation}
  According to \cite[Proposition 3.1]{Met} (see also \cite{CMPR}),  we define the domain of $\Lb$ in the following way:  $u\in D(\Lb)$ if, and only if, $u\in C([0,M])\cap C^2(]0,M])$, $u'(M)=0$ and
\begin{eqnarray}\label{dom:1} 
& &\lim_{x\to 0^+}\Lb u(x)=0 \ \mbox{ if } b= 0,\\
& & \label{dom:2} u\in C^1([0,\delta])\ \mbox{and} \  \lim_{x\to 0^+} xu''(x)=0 \ \mbox{if}\ b>0.
\end{eqnarray}
It is known  from \cite{Met,CM,CMPR,CT} that
the operator $\Lb$ with domain $D(\Lb)$ generates a \textit{bounded   analytic $C_0$--semigroup $(T(t))_{t\geq 0}$ of positive contractions and angle $\pi/2$ on $C([0,M])$}.

We   are here interested in proving 
  estimates  for the norm of the resolvent  operators of $\Lb$ and of their gradient with   constants which  depend only on $B$. In order to do this we need the following fact.

\begin{remark}\label{r1-laplace} \rm
 Let $B,\ \gamma_0>0$. For every $b\in [0,B]$ and $\gamma\in\R$, $\gamma\geq \gamma_0$  
consider the one--dimensional second order differential operator 
\begin{equation}\label{e.1-operatorl}
G^{\gamma,b}u(x)=\gamma u''(x)+bu'(x), \quad x\in [0,\infty),
\end{equation}
with  domain  $D:=\{u\in C^2([0,\infty]):\ u'(0)=0\}$ and $\gamma\geq \gamma_0$, $b\in [0,B]$. 
It is well known that the operator $(G^{\gamma,b},D)$ generates a bounded analytic semigroup of angle $\pi/2$ in $C([0,\infty])$, see, e.g., \cite[Theorem VI.4.3]{EN}. In particular, $(G^{\gamma,b},D)$ satisfies the following properties:

\textit{There exists $c_1,\, c_2, R>0$ depending only on $B$ and on $\gamma_0$ such that, for every $\lambda\in\C$ with ${\rm Re}\lambda>R$ and $u\in C([0,\infty])$, 
\begin{eqnarray}\label{ea-1}
& &\|R(\lambda, G^{\gamma,b})\|\leq \frac{c_1}{|\lambda|}\\
& & \|(R(\lambda, G^{\gamma,b})u)'\|_\infty\leq \frac{c_2}{\sqrt{|\lambda|}}\|u\|_\infty.
\end{eqnarray}}

The proof is as follows.

Set $G:=G^{1,0}$. Then, for every $\lambda=|\lambda|e^{i\theta}\not \in (-\infty,0]$ with $|\theta|<\pi$, we have
\begin{equation}\label{ea-0}
R(\lambda, G)u=\frac{1}{2\mu}\int_{0}^\infty e^{-\mu|x-s|}u(s)\,ds+c e^{-\mu x}, \quad u\in  C([0,\infty]),
\end{equation}
where $\mu^2=\lambda$ with ${\rm Re}\mu>0$ and $c=\frac{1}{2\mu}\int_{0}^\infty e^{-\mu s}u(s)\, ds$, see, e.g., \cite[Theorem VI.4.3, Theorem 4.2]{EN}. So, from \eqref{ea-0} it follows that 
\begin{eqnarray}
\label{ea-2} & & \|R(\lambda, G)\|\leq \frac{3}{2|\lambda|\cos(\theta/2)}\\
\label{ea-3} & & \|(R(\lambda, G)u)'\|_\infty\leq \frac{1}{\sqrt{|\lambda|}\cos(\theta/2)}\|u\|_\infty,
\end{eqnarray}
for every $\lambda=|\lambda|e^{i\theta}\not \in (-\infty,0]$ with $|\theta|<\pi$ and $u\in C([0,\infty])$.

We now consider the operator $G^{\gamma,0}$ with $\gamma\geq \gamma_0$ and observe that
\[
R(\lambda, G^{\gamma,0})=\gamma^{-1}R(\lambda/\gamma,G),\quad \lambda\in\C\setminus (-\infty,0].
\]
This equality implies via \eqref{ea-2} and \eqref{ea-3} that
\begin{eqnarray}
\label{ea-4} & & \|R(\lambda, G^{\gamma,0})\|\leq \gamma^{-1}\frac{3}{2|\lambda/\gamma|\cos(\theta/2)}=\frac{3}{2|\lambda|\cos(\theta/2)}\\
\label{ea-5} & & \|(R(\lambda, G^{\gamma,0})u)'\|_\infty\leq \gamma^{-1}\frac{1}{\sqrt{|\lambda|/\gamma}\cos(\theta/2)}\|u\|_\infty \\
& & \quad \leq \frac{1}{\sqrt{\gamma_0|\lambda|}\cos(\theta/2)}\|u\|_\infty,\nonumber
\end{eqnarray}
for every $\lambda=|\lambda|e^{i\theta}\not \in (-\infty,0]$ with $|\theta|<\pi$ and $u\in C([0,\infty])$.

If we set $H^bu=bu'$ for $b\in [0,B]$  and $u\in C^1([0,\infty])$, then by \eqref{ea-5} we obtain,  for every $\lambda=|\lambda|e^{i\theta}\not \in (-\infty,0]$ with $|\theta|<\pi$ and $u\in C([0,\infty])$, that
\begin{equation}\label{ea-6}
\|H^bR(\lambda, G^{\gamma,0})u\|_\infty\leq \frac{B}{\sqrt{\gamma_0|\lambda|}\cos(\theta/2)}\|u\|_\infty. 
\end{equation}
By \eqref{ea-6}, for every $\lambda=|\lambda|e^{i\theta}$ with $|\theta|<\pi/2$ and $|\lambda|>R=8B^2/\gamma_0$ and $b\in [0,B]$, the operator $H^bR(\lambda, G^{\gamma,0})$ has norm $<1/2$ and so the operator $S_b(\lambda):=I-H^bR(\lambda, G^{\gamma,0})$ is invertible with inverse 
\begin{equation}\label{ea-7}
(S_b(\lambda))^{-1}=\sum_{n=1}^\infty[H^bR(\lambda,G^{\gamma,0})]^n
\end{equation}
 so that $\|(S_b(\lambda))^{-1}\|\leq 2$. This estimate, combined  with the identity $\lambda-G^{\gamma, b}=\lambda-G^{\gamma,0}-H^b=[I-H^bR(\lambda,G^{\gamma,0})](\lambda-G^{\gamma,0})$ implies, for every $\lambda=|\lambda|e^{i\theta}$ with $|\theta|<\pi/2$ and $|\lambda|>R=8B^2/\gamma_0$ and $b\in [0,B]$, that
\begin{equation}\label{ea-8}
R(\lambda,G^{\gamma,b})=R(\lambda,G^{\gamma,0})(S_b(\lambda))^{-1}.
\end{equation}
So, by \eqref{ea-4}, \eqref{ea-5} and \eqref{ea-8} we obtain,  for every $\lambda=|\lambda|e^{i\theta}$ with $|\theta|<\pi/2$,  $|\lambda|>R=8B^2/\gamma_0$, $u\in C([0,\infty])$ and $b\in [0,B]$, that
\begin{eqnarray}
\label{ea-9} & & \|R(\lambda, G^{\gamma,b})\|\leq \frac{3}{|\lambda|\cos(\theta/2)}\\
\label{ea-10} & & \|(R(\lambda, G^{\gamma,b})u)'\|_\infty\leq  \frac{2}{\sqrt{\gamma_0|\lambda|}\cos(\theta/2)}\|u\|_\infty.
\end{eqnarray}\qed
\end{remark}

\begin{prop}\label{p.1-gradiente} Let $B, M>0$ and  let $\gamma\in C([0,M])$ be a strictly positive function with $\gamma_0:=\min_{x\in [0,M]}\gamma(x)$. Then, for every $b\in [0,B]$, the following properties hold.
\begin{enumerate} 
\item There exist  $d_0=d_0(B, \gamma), R=R(B, \gamma)>0$ such that, for every $u\in C([0,M])$ and for every $\lambda\in\C$ with ${\rm Re}\lambda >R$, we have 
\begin{eqnarray}
 ||R(\lambda,\Lb) u||_\infty & \leq & d_0 \frac{||u||_\infty}{|\lambda|},\\
 ||\sqrt{x}(R(\lambda,\Lb) u)'||_\infty &\leq & d_0 \frac{||u||_\infty}{\sqrt{|\lambda|}}.
\end{eqnarray}
Moreover, $\lim_{x\to 0^+}\sqrt{x}(R(\lambda,\Lb) u)'(x)=0$. In particular, $\sqrt{x}(R(\lambda,\Lb) u)'$ extends continuously to $[0,M]$.

\item If $(T(t))_{t\geq 0}$ is the semigroup generated by $(\Lb, D(\Lb))$, then there exist $K=K(B,\gamma), \alpha= \alpha(B, \gamma)>0$ such that, for every $u\in C([0,M])$, we have 
\begin{eqnarray}
& & ||t\Lb T(t)|| \leq Ke^{\alpha t},\qquad t\geq 0\\
& &||\sqrt{x}(T(t)u)'||_\infty\leq \frac{K e^{\alpha t}}{\sqrt t}||u||_\infty,        \qquad  0<t<R^{-1},\\
& &||\sqrt{x}(T(t)u)'||_\infty\leq K e^{\alpha t}||u||_\infty, \qquad  t\geq R^{-1},
\end{eqnarray}
where $R$ is the same constant which appears in  part {\rm (1)}.

Moreover, $\lim_{x\to 0^+} \sqrt{x}(T(t)u)'(x)=0$ if $t>0$. In particular, $\sqrt{x}(T(t)u)'$ extends continuously to $[0,M]$ if $t>0$.
\end{enumerate}
\end{prop}

\begin{proof} 
W.l.o.g. we may  assume $M=1$.

(1) For each $n\in\N$ and $i\in \{1,\ldots, n-1\}$ set $I^i_n=\left[\frac{i-1}{n},\frac{i+1}{n}\right]$ and  let $\{\varphi^i_n\}_{i=1}^{n-1}\subset C^\infty(\R)$ such that $\sum_{i=1}^{n-1}(\varphi^i_n)^2\equiv 1$ on $[0,1]$, 
${\rm supp} (\varphi^i_n)\subset \left[\frac{i-1}{n},\frac{i+1}{n}\right]$ for $i=2,\ldots, n-2$, ${\rm supp} (\varphi^1_n)\subset \left]-\infty, \frac{2}{n}\right]$ and ${\rm supp} (\varphi^{n-1}_n)\subset \left[\frac{n-2}{n}, \infty\right[$. Observe that if $v=\sum_{i=1}^{n-1} v_i\varphi^i_n$ with $v_i\in C([0,1])$ for $i=1,\ldots,n$, then 
\begin{equation}\label{eq:v3} 
||v||_\infty\leq 3 \sup_{i=1, \ldots,n-1}||v_i||_\infty.
\end{equation}
For every  $i\in\{1, \ldots,n-2\}$ we   consider the operators
\[ 
L^i_n u =  \gamma\left(\frac i n\right)xu''(x)+bu'(x),\quad u\in D(L^i_n), 
\]
with domain $D(L^i_n)$ defined as follows: if  $b=0$ 
\[
D(L^i_n)=
\{ u\in C([0,\infty]) \cap C^2(]0,\infty[)\, \mid\,  \lim_{x\to 0^+} L^i_n u(x) =0; 
   \lim_{x\to +\infty}L^i_n u(x)=0\}, 
	\]
if $b>0$ 
\[	
D(L^i_n)=\{ u\in C^1([0,\infty[)\cap C^2(]0,\infty[)\cap C([0,\infty])\,\mid\, 
\lim_{x\to 0^+} xu''(x)=0,\ \lim_{x\to + \infty} L^i_n u(x)=0 \}. 
\]
For $i=n-1$ we consider the operator
\[
L^{n-1}_n u =  \gamma\left(\frac{ n-1}{n}\right)u''(x)+bu'(x),\quad u\in D(L^{n-1}_n),
\]
with domain
\[
D(L^{n-1}_n)=\{u\in C^2([-\infty,1])\, \mid\, u'(1)=0\}.
\]
By \cite[Corollary 4.2] {AM-3} 
and Remark \ref{r1-laplace}, there exists $d_1=d_1(B, \gamma_0)>0, R=R(B,\gamma_0)>0$ such that, for every  $\lambda\in\C$ with ${\rm Re}\lambda>R$,  we have
\begin{equation}\label{e.star}
||R(\lambda, L^i_n)|| \leq \frac{d_1}{|\lambda|},\quad n\in\N,\, i=1,\ldots, n-1.
\end{equation}
Fix $\lambda\in\C$, with ${\rm Re}\lambda>R$. For each $n\in\N$ let $S_n(\lambda)\colon C([0,1])\to C([0,1])$ be the operator defined by
\[ 
S_n(\lambda) f= \sum_{i=1}^{n-1} \varphi^i_n R(\lambda, L^i_n)(\varphi^i_n f),\quad f\in C([0,1]).
\]
By \eqref{eq:v3} and \eqref{e.star} we get, for every  $n\in\N$, that
\[ 
||S_n(\lambda)f||_\infty \leq 3 \sup_{i=1,\ldots,n-1} ||R(\lambda, L^i_n)(\varphi^i_n f)|| \leq \frac{3d_1}{|\lambda|}||f||_\infty,\quad f\in C([0,1]).
\]
Since $R(\lambda, L^i_n)(\varphi^i_n f)\in D(L^i_n)$ for every $i=1,\ldots, n-1$ and $f\in C([0,1])$, $\varphi^{n-1}_n\equiv 0$ and $\varphi^{n-1}_n\equiv 1$ in an neighbourhood of $0$ and in an neighbourhood of $1$ respectively, we have $\varphi^i_nR(\lambda, L^i_n)(\varphi^i_n f)\in D(L^{\gamma,b})$ and so we can consider $(\lambda-\Lb)(S_n(\lambda)f)$ for every $f\in C([0,1])$. In particular, a straightforward calculation gives
\begin{eqnarray*}
& &(\lambda-\Lb)(S_n(\lambda)f)= f + \sum_{i=1}^{n-1} \varphi^i_n (L^i_n-\Lb)R(\lambda, L^i_n)(\varphi^i_n f)\\
& & - \sum_{i=1}^{n-1} \Lb(\varphi^i_n) R(\lambda, L^i_n)(\varphi^i_n f) 
-2 \gamma(x) \sum_{i=1}^{n-1}(\varphi^i_n)' x(R(\lambda, L^i_n)(\varphi^i_n f))' \\
& & =f + C_1^n(\lambda)f + C_2^n(\lambda) f + C_3^n(\lambda) f,\quad f\in C([0,1]),\ n\in\N.
\end{eqnarray*}
Applying again \eqref{eq:v3} and \eqref{e.star} we obtain 
\begin{equation}\label{e.c2}
||C_2^n(\lambda)f||_\infty\leq 3||f||_\infty \frac{d_1}{|\lambda|} \sup_{i=1, \ldots,n-1} ||\Lb(\varphi^i_n)||_\infty, \quad f\in C([0,1]),\, n\in\N.
\end{equation}
On the other hand, by  \cite[Proposition 5.1(2)]{AM-3}, Remark \ref{r1-laplace} and \eqref{eq:v3}, there exists $d_2=d_2(B, \gamma_0)>0$ such that 
\begin{equation}\label{e.c3}
 ||C_3^n(\lambda) f||_\infty \leq 3d_2 ||\gamma||_\infty \sup_{i=1,\ldots,n-1}||(\varphi^i_n) '||_\infty \frac{||f||_\infty}{\sqrt{|\lambda|}}, \quad f\in C([0,1]),\ n\in\N. 
\end{equation}
In order to estimate $||C_1^n(\lambda)||$, we  observe, for every  $n\in\N$, that 
\begin{eqnarray*}
& &\varphi^i_n(L^i_n-\Lb)R(\lambda, L^i_n) (\varphi^i_nf) =\\
& & \qquad =\left\{\begin{array}{ll}
 \varphi^i_n \left[\gamma\left(\frac i n\right)-\gamma(x)\right]x(R(\lambda, L^i_n) (\varphi^i_nf))''(x) & \mbox{if $i=1,\ldots, n-2$,}\\
\varphi^{n-1}_n \left[\gamma\left(\frac{n-1}{n}\right)-\gamma(x)x\right](R(\lambda, L^{n-1}_n) (\varphi^{n-1}_nf))''(x) & \mbox{if $i= n-1$}.
\end{array}\right.
\end{eqnarray*}
Fixed $\varepsilon>0$, we now choose $\delta>0$ so  that $|\gamma(x)-\gamma (y)|<\varepsilon$ if $|x-y|<\delta$ and that $|\gamma(x)-\gamma (y)|+\Gamma_0|1-x|<\varepsilon$ if $x,y\in [1-\delta, 1]$, where $\Gamma_0:=\max_{x\in [0,1]}\gamma(x)$. If we take   $\overline n\in\N$ such that $\frac{2}{\overline n}<\delta$, 
then we have that  $|\gamma(x)-\gamma(\frac{i}{\overline n})|<\varepsilon$ if $x\in I^i_{\overline n}$ and $i\in\{1,\ldots, \overline{n}-2\}$ and that $|\gamma\left(\frac{\overline{n}-1}{\overline n}\right)-\gamma(x)x|<\varepsilon$ if $x\in I_{\overline n}^{\overline{n}-1}$. 
So,  it follows from \cite[Proposition 5.1(2)]{AM-3}, Remark \ref{r1-laplace} and \eqref{eq:v3} that 
\begin{equation}\label{e.c1}
||C_1^{\overline n}(\lambda)f||_\infty \leq 3 \varepsilon \frac{1}{\gamma_0}\left(1+d_1 +  B \frac{d_1}{|\lambda|}\right) ||f||_\infty, \quad f\in C([0,1]).
\end{equation}
 Therefore, combining \eqref{e.c2}, \eqref{e.c3} and \eqref{e.c1}, we obtain 
\begin{eqnarray*}
& & ||C_1^{\overline n}(\lambda)||+||C_2^{\overline n}(\lambda)||+||C_3^{\overline n}(\lambda)||\leq \frac{d_1}{|\lambda|} \sup_{ i=1, \ldots,\overline{n}-1} ||\Lb(\varphi^i_{\overline n})||_\infty +\\
& & + 3d_2 ||\gamma||_\infty \sup_{i=1,\ldots,\overline{n}-1}||(\varphi^i_{\overline{n}})'||_\infty \frac{1}{\sqrt{|\lambda|}}+ 
3 \varepsilon \frac{1}{\gamma_0}\left(1+d_1 +  B \frac{d_1}{|\lambda|}\right). 
\end{eqnarray*}
Now, let $\varepsilon >0$ be small enough that $3\varepsilon \frac{1+d_1}{\gamma_0}<1/4$ and $R'>R$ such that
\[ 
\frac{d_1}{|\lambda|} \sup_{ i=1, \ldots,\overline n-1} ||L(\varphi^i_{\overline n})||_\infty + 3d_2 ||\gamma||_\infty \sup_{i=1,\ldots,\overline n-1}||(\varphi^i_{\overline n})'||_\infty  \frac{1}{\sqrt{|\lambda|}}+ 
 B \frac{d_1}{|\lambda|} <\frac 1 4 
 \]
for every  $\lambda\in C\setminus [0,+\infty)$ with $|\lambda|\geq R'$ (in particular,  with ${\rm Re}\lambda\, >R'$).  So,  $R'=R'(\gamma_0, B)$. 
Since 
  \[ 
	\|C_{1}^{\overline{n}}(\lambda)+C_{2}^{\overline{n}}(\lambda) +C_{3}^{\overline{n}}(\lambda)\|<1/2,
	\]
   the operator $B(\lambda)=(\lambda-\Lb)S_{\overline{n}}(\lambda)$ is invertible in $C([0,1])$ with  $\|(B(\lambda))^{-1}\|\leq 2$. So, for every $\lambda\in \C$ with ${\rm Re}\lambda> R'$,  we have $R(\lambda, \Lb)=S_{\overline n}(\lambda)B(\lambda)^{-1}$ and 
\begin{equation}\label{e.dis-9}
\|R(\lambda, \Lb)\|\leq \frac{6 d_1 }{|\lambda|}.
\end{equation}
Fixed  $\lambda\in \C$ with ${\rm Re}\lambda> R'$, it follows via  \cite[Proposition 5.1(2)]{AM-3}, Remark \ref{r1-laplace} and \eqref{eq:v3} that, for every $u\in C([0,1])$, we have
\begin{eqnarray*}
& & ||\sqrt{x} (R(\lambda, \Lb)u)'||=||\sqrt{x} (S_{\overline n}(\lambda)B(\lambda)^{-1}u)'||\\
 & & \leq ||\sum_{i=1}^{\overline{n}-1} (\varphi^i_{\overline n})' R(\lambda, L^i_{\overline n})(\varphi^i_{\overline n}B(\lambda)^{-1} u)||_\infty 
+ ||\sum_{i=1}^{\overline{n}-1} \varphi^i_{\overline n} \sqrt{x}[R(\lambda, L^i_{\overline n})(\varphi^i_{\overline n}B(\lambda)^{-1} u)]'||_\infty \\
& & \leq \left(\frac{18\overline{n} d_1 \sup_{i=1,\ldots,  {\overline n}-1}||(\varphi^i_{\overline n})'||_\infty}{|\lambda|} + \frac{3\overline{n} d_2\sup_{i=1,\ldots,  {\overline n}-1}||\varphi^i_{\overline n}||_\infty}{\sqrt{|\lambda|}}\right)||u||_\infty.
\end{eqnarray*} 
If we choose $d_0= \max\{ 18 \overline{n} d_1 \sup_{i=1,\ldots, \overline{n}}||(\varphi^i_{\overline n})'||_\infty + 3 \overline{n}d_2\sup_{i=1,\ldots,  {\overline n}-1}||\varphi^i_{\overline n}||_\infty, 6d_1\}$, then the thesis now follows.  
We also have
\begin{eqnarray*} 
\lim_{x\to 0^+}\sqrt{x} (R(\lambda, \Lb)u)'(x)= \lim_{x\to 0}\sqrt{x}\left( \sum_{i=1}^{\overline{n}-1} (\varphi^i_{\overline n})' R(\lambda, L^i_{\overline n})(\varphi^i_{\overline n}B(\lambda)^{-1} u)  + \right.\\
\left.+ \sum_{i=1}^{\overline{n}-1} (\varphi^i_{\overline n})' R(\lambda, L^i_{\overline n})(\varphi^i_{\overline n}B(\lambda)^{-1} u)'\right)=0,\end{eqnarray*}
by   \cite[Propositions 5.1(2) and 5.2]{AM-3} and Remark \ref{r1-laplace}.

(2) Since the resolvent operators of $\Lb$ satisfy the part (1) of this proposition, we can apply \cite[Proposition 2.1.11]{L} to conclude that,   for every $\lambda\in\C$ with $\lambda\not=R$ and $|\ar(\lambda-R)|<\pi-\arctan{d_0}$, we have
\[
 \|R(\lambda, \Lb)\|\leq \frac{\widetilde{d_0} }{|\lambda-R|},
\]
where $\widetilde{d}_0= 2d_0(1/(4d_0^2)+1)^{-1/2}$. Then 
 there exist $K=K(B,\gamma)>0$  and $\alpha= \alpha(B,  \gamma)$ such that
\[
||t(\Lb-\alpha)T(t)|| \leq Ke^{\alpha t},\quad t\geq 0
\]
(see, f.i.,   \cite[Proposition 2.1.1]{L} and also the estimates in the relative proof).
Since $(T(t))_{t\geq 0}$  contractive, it  follows that
\[
 ||tLT(t)||\leq (K+1)e^{\alpha t},\quad t\geq 0.
\]
Finally, if $u\in D(\Lb)$, then part (1) of this proposition  ensures that, for every $\lambda\in\R$, $\lambda>R$,  we have 
\[
 ||\sqrt{x}u'||_\infty \leq \frac{d_0}{\sqrt{\lambda}}||\lambda u - \Lb u||_\infty.
\]
As  the semigroup $(T(t))_{t\geq  0}$ is  analytic and hence,  $T(t) f\in D(\Lb)$ for every $f\in C[0,1]$ and $t>0$, it follows that
\[ 
||\sqrt{x}(T(t)f)'||_\infty \leq \frac{d_0}{\sqrt{\lambda}}||\lambda T(t)f - \Lb T(t)f||_\infty \leq \left(d_0\sqrt{\lambda} + \frac{d_0Ke^{\alpha t}}{t\sqrt\lambda}\right)||f||_\infty,
\]
for every $f\in C[0,1]$ and $t>0$. So, if we choose $\lambda=t^{-1}$ for every $t<R^{-1}$ and $\lambda=R+1$ for every $t\geq R^{-1}$, then we get the assertion.  
Moreover, $\lim_{x\to 0^+}\sqrt{x}(T(t) f(x))'=0$,  for $t>0$.  Indeed, this property holds for every $u\in D(\Lb)$ by part (1) of this proposition.
\end{proof}

\begin{remark}\rm Since the operator $(\Lb,D(\Lb))$ generates a bounded   analytic $C_0$--semigroup $(T(t))_{t\geq 0}$ of positive contractions and angle $\pi/2$ on $C([0,M])$, for every $\theta\in (\pi/2,\pi)$ there exists $M_0=M_0(\theta)>0$ such that $\|R(\lambda,\Lb)\|\leq M_0/|\lambda|$ for all $\lambda\in \C\setminus\{0\}$ with $|{\rm arg}(\lambda)|<\theta$. Moreover, there exists $M_1>0$ such that $\|t\Lb T(t)\|\leq M_1$ for every $t\geq 0$. But,  the constants $M_0$ and $M_1$ can depend on the functions $b$ and $\gamma$.

\end{remark}

\begin{corollary}\label{c.1-gradiente}  Let $B, M>0$ and  let $\gamma\in C([0,M])$ be a strictly positive function. Then there exist $\overline{\varepsilon}>0$, $C>0$ and $D>0$  depending only on $B$ and $\gamma$ such that,  for every $0<\varepsilon<\overline{\varepsilon}$, $b\in [0,M]$  and    
$u\in D(\Lb)$, we have 
\[ 
\|\sqrt{x} u'\|_\infty\leq \frac{C} {\varepsilon}  \|u\|_\infty+ D\varepsilon \|\Lb u\|_\infty.
 \]
\end{corollary}

\begin{proof}
Fix $u\in D(\Lb)$ and $\lambda\in \C$ with ${\rm Re}\lambda >R$, where $R$ is the costant which appears in Proposition \ref{p.1-gradiente}(1).  Then  there exists $v\in C([0,M])$ such that $R(\lambda, \Lb)v=u$ and hence, by Propositon \ref{p.1-gradiente}, we have that 
\begin{eqnarray}\label{e.dimen}
\|\sqrt{x}u'\|_\infty&=&\|\sqrt{x}(R(\lambda, L^{\gamma,b})v)'\|_\infty\leq \frac{d_0}{\sqrt{|\lambda|}}\|\lambda u-L^{\gamma,b}u\|_\infty\nonumber\\
&\leq & d_0\left(\sqrt{|\lambda|}\|u\|_\infty+\frac{1}{\sqrt{|\lambda|}}\|L^{\gamma,b}u\|_\infty\right),
\end{eqnarray}
where $d_0$ depends only on $B$ and $\gamma$.  Now, the assertion follows from  \eqref{e.dimen} and from Proposition \ref{p.1-gradiente}(1) by choosing $\overline\varepsilon =R^{-1}$ and, for $0<\varepsilon<\overline\varepsilon$,  $\sqrt{|\lambda|}=1/\varepsilon$.
\end{proof}

Set $C^2_\diamond([0,M])=\{u\in C^2([0,M]):\ u'(M)=0\}$. Then

\begin{prop}\label{p.1-core} Let  $b\geq 0$ and let $\gamma \in C([0,M])$ be a  strictly positive function. Then the  space $C^2_\diamond([0,M])$ is a core for the operator  $\Lb$ with domain $D(\Lb)$ defined according to  \eqref{dom:1} if $b=0$ or to \eqref{dom:2} if $b>0$.
\end{prop}

\begin{proof} 
 The assertion follows with the same argument of Proposition 3.1 in \cite{AM-3}, with some minor chages.
\end{proof}


\begin{remark}\label{r.compatezza}\rm The inclusion $(D(\Lb), \|\ \|_{\Lb})\hookrightarrow C([0,M])$ is compact (here, $\|\ \|_{\Lb}$ denotes the graph norm), see \cite[Theorem 4.1]{Met} or \cite[Lemma 3.2]{CM}. So $(\Lb,D(\Lb))$ has compact resolvent, \cite[Proposition 4.25]{EN}. Since $(\Lb,D(\Lb))$ generates a bounded analytic $C_0$--semigroup $(T(t))_{t\geq 0}$ on $C([0,M])$ (and hence, a norm continuous $C_0$--semigroup) and has compact resolvent, $(T(t))_{t\geq 0}$ is also compact, \cite[Theorem 4.29]{EN}.
\end{remark}


\section{The $d$-dimensional case with constant  drift term}

Set $Q^d=[0,M]^d$ and, for each $i=1,\ldots,d$, define $\partial (Q^d)_i:=\{x\in Q^d\, \mid \, x_i=0\}$ and $\partial (Q^d)^i:=\{x\in Q^d\, \mid \, x_i=M\}$.
Let $B>0$ and fix $b=(b_1,b_2,\ldots,b_d)\in [0,B]^d$ and $\gamma=(\gamma_1,\gamma_2,\ldots,\gamma_d)\in C([0,M])^d$, with each $\gamma_i$  strictly positive.

For each $i\in\{1,\ldots,d\}$  set $L^{\gamma_i,b_i}=\gamma_i(x_i) x_i\partial^2_{x_i}+b_i\partial_{x_i}$,  with domain $D(L^{\gamma_i, b_i})$ defined according to \eqref{dom:1} if $b_i=0$ or to \eqref{dom:2} if $b_i>0$.

We know that each operator $(L^{\gamma_i,b_i},D(L^{\gamma_i,b_i}))$ generates a bounded analytic  compact $C_0$-semigroup $(T_i(t))_{t\geq 0}$ of positive  contractions and of angle $\pi/2$ in $C([0,M])$.  So,  the injective tensor product $(T(t))_{t\geq 0}=(\hat{\otimes}_{\epsilon, i=1}^d T_i(t))_{t\geq 0}$ is also a  \textit{bounded analytic  compact $C_0$-semigroup of positive contractions and of angle $\pi/2$} in  $C([0,M]^d)=\hat{\otimes}_{d,\epsilon} C([0,M])$, see \cite[Proposition, p.23, and p.24]{N}, \cite[Appendix A]{CC1} (see also \cite[\S 2.2]{AM-2}). In particular, the infinitesimal generator $(\cL^{\gamma,b}, D(\cL^{\gamma, b}))$ of  $(T(t))_{t\geq 0}$ is the closure of the operator  
\begin{equation}\label{eq:tensor}
L^{\gamma,b} = \sum_{i=1}^d  L^{\gamma_i,b_i}\otimes \left(\otimes_{j\not=i}I_{x_j}\right),
\end{equation}
with domain  $\otimes_{i=1}^dD(L^{\gamma_i,b_i})$,
where $I_{x_j}$  denote the identity map acting in  $C([0,M])$  with respect to the variable $x_j$. Clearly, for every $u\in \otimes_{i=1}^dD(L^{\gamma_i,b_i})$, we have
\[
\cL^{\gamma,b} u(x) =\sum_{i=1}^d \gamma_i(x_i) x_i\partial_{x_i}^2u + b_i\partial_{x_i}u .
\]
Moreover, if we  define  $C_{\diamond}^2(Q^d)=\cap_{i=1}^d\{u\in C^2(Q^d):\ \forall x\in \partial(Q^d)^i\ \partial_{x_i}u(x)=0  \}$ (such a space is a Banach space when endowed with the supremum norm $\|\ \|_{2,\infty}$),  then the  following holds.

\begin{prop}\label{p.-dcore} Let $b=(b_1,b_2,\ldots,b_d)\in [0,\infty[^d$ and $\gamma=(\gamma_1,\gamma_2,\ldots,\gamma_d)$, with each $\gamma_i$  a strictly positive continuous function on $[0,M]$. Then the space $C_{\diamond}^2(Q^d)$ is a core for the operator  $(\cL^{\gamma, b}, D(\cL^{\gamma, b}))$. 
\end{prop}

\begin{proof} By Proposition \ref{p.1-core} the space $C^2_\diamond ([0,M])$ is  a core for the one--dimensional operator $(L^{\gamma_i, b_i},D(L^{\gamma_i, b_i}))$ for every $i=1,\dots,d$. So,  $\otimes_{i=1}^d C^2_\diamond ([0,M])$ is a core for the operator $(\cL^{\gamma,b},D(\cL^{\gamma,b}))$.
On the other hand, it is known that $\otimes_{i=1}^d C^2_\diamond ([0,M])$ is dense in $C_\diamond^2(Q^d)$ with respect to the $C^2$-norm which is clearly stronger than the graph-norm of $\cL^{\gamma,b}$. So, 
it follows that $C_\diamond^2(Q^d)$ is a subspace of the domain of the closure of $\cL^{\gamma,b}$ and is dense therein with respect to the graph norm. 
\end{proof}


We now prove  that  the operator $(\cL^{\gamma,b},D(\cL^{\gamma,b}))$ also shares similar gradient estimates with  the analogous one--dimensional operator.

\begin{prop}\label{p.dimd-grad}   Let $B>0$ and $\gamma=(\gamma_1,\gamma_2,\ldots,\gamma_d)\in (C([0,M]))^d$, with each $\gamma_i$  strictly positive. Then, for every $b\in [0,B]^d$,  the following properties hold.
\begin{enumerate}
\item There exist $K, \alpha, \overline t>0$ depending on $B$ and on $\gamma$ such that, for every $u\in C(Q^d)$ and $i=1,\ldots,d$, we have
\begin{eqnarray}
\label{eq:s1} & & ||t\cL^{\gamma,b} T(t)|| \leq Ke^{\alpha t},\quad t\geq 0.\\
\label{eq:s2} & &||\sqrt{x_i}\partial_{x_i}(T(t)u)||_\infty\leq \frac{K e^{\alpha t}}{\sqrt t}||u||_\infty,        \quad 0< t<\overline t.\\
\label{eq:s3} & &||\sqrt{x_i}\partial_{x_i}(T(t)u)||_\infty\leq K e^{\alpha t}||u||_\infty, \qquad  t\geq \overline t.
\end{eqnarray}
Moreover, for every $i\in \{1,\ldots,d\}$ and  $u\in C(Q^d)$, $\sqrt{x_i}\partial_{x_i}(T(t)u)\in C(Q^d)$ and 
\begin{equation}\label{limsem}
 \lim_{x_i\to 0^+} \sup_{x_j\in [0,M], j\in \{1,\ldots,d\}\setminus\{i\}}\sqrt{x_i}\partial_{x_i}(T(t)u)=0.
\end{equation}
\item There exist  $d_1, d_2, R>0$ depending on $B$ and on $\gamma$  such that, for every $\lambda\in \C$ with  {\rm Re}$\lambda>R$ , $u\in C(Q^d)$ and $i=1,\ldots,d$, we have
\begin{eqnarray}
\label{eq:r1} & &||R(\lambda,\cL^{\gamma,b}) u||_\infty \leq d_1 \frac{||u||_\infty}{|\lambda|},\\
\label{eq:r2} & & ||\sqrt{x_i}\partial_{x_i}(R(\lambda,\cL^{\gamma,b}) u)||_\infty \leq d_2 \frac{||u||_\infty}{\sqrt{|\lambda|}}.
 \end{eqnarray}
Moreover, for every $i\in \{1,\ldots,d\}$  and $u\in C(Q^d)$,  $\sqrt{x_i}\partial_{x_i}(R(\lambda,\cL^{\gamma,b}) u)\in C(Q^d)$ and 
\begin{equation}\label{limres} 
\lim_{x_i\to 0^+} \sup_{x_j\in [0,M],j\in \{1,\ldots,d\}\setminus\{i\}}\sqrt{x_i}\partial_{x_i}(R(\lambda,\cL^{\gamma,b}) u)(x)=0.
\end{equation}
\item There exist  $C, D, \overline \varepsilon>0$  depending on $B$ and on $\gamma$ such that,  for every $0<\varepsilon<\overline\varepsilon$,  $i=1,\ldots, d$  and    
$u\in D(\cL^{\gamma,b})$, we have 
\[ 
\|\sqrt{x_i} \partial_{x_i}u\|_\infty\leq \frac{C} {\varepsilon}  \|u\|_\infty+ D\varepsilon \|\cL^{\gamma,b}u\|_\infty.
 \]
\end{enumerate}
\end{prop}

\begin{proof} 
(1)  By Proposition \ref{p.1-gradiente}(2) there exists $\ov{t},K, \alpha >0$ depending on $B$ and $\gamma_i$ such that the  operators  $\sqrt{x_i}\partial_{x_i} T_i(t)$, for $i=1,\ldots,d$,   are   bounded  on $C([0,M])$  with norm less or equal to ${K}e^{\alpha t}/\sqrt{t}$ if  $0<t< \ov{t}$ and to $Ke^{\alpha t}$ if $t\geq \ov{t}$. Then the  operators  
\[
\sqrt{x_i}\partial_{x_i}T(t)=T_1(t)\widehat\otimes_\varepsilon\ldots \widehat\otimes_\varepsilon(\sqrt{x_i}\partial_{x_i }T_i(t))\widehat\otimes_\varepsilon \ldots \widehat\otimes_\varepsilon T_d(t),\quad i=1,\ldots,d,
\] 
are  also  bounded  on $C(Q^d)$ with norm less or equal to $Ke^{\alpha t}/\sqrt{t}$ if $0<t<\ov{t}$ or to $Ke^{\alpha t}$ if $t\geq \ov{t}$, \cite{J} (see also \cite[Appendix A]{CC1} or \cite[\S 2.2]{AM-2}). So, inequalities \eqref{eq:s2} and \eqref{eq:s3} are satisfied.
In particular, for every $u\in C(Q^d)$ and $i=1,\ldots,d$ we have $\sqrt{x_i}\partial_{x_i}T(t)u\in C(Q^d)$. 

Moreover, again by Proposition \ref{p.1-gradiente}(2) we have that $\|tL^{\gamma_i,b_i}T_i(t)\|\leq Ke^{\alpha t}$ for every $t\geq 0$. Then, via \eqref{eq:tensor} the linear operators
\[
tL^{\gamma,b}T(t)=\sum_{i=1}^d tL^{\gamma_i,b_i}T_i(t)\otimes (\otimes_{j\not=i}T_j(t))
\]
are bounded on $\otimes_{i=1}^dC([0,M])$ with norm  less or equal to $dKe^{\alpha t}$ for every $t\geq 0$. So, \eqref{eq:s1} is satisfied on $\otimes_{i=1}^dC([0,M])$. By the density of $\otimes_{i=1}^dC([0,M])$ in $C(Q^d)$  and the continuity of the linear operators $tL^{\gamma,b}T(t)$ in $C(Q^d)$ (recall that $(T(t))_{t\geq 0}$ is an analytic $C_0$--semigroup in $C(Q^d)$) it follows that \eqref{eq:s1} is satisfied.

Finally, if $u\in \otimes_{i=1}^dC([0,M])$,  then \eqref{limsem} is clearly satisfied by  Proposition \ref{p.1-gradiente}(2). The density of $\otimes_{i=1}^dC([0,M])$ in $C(Q^d)$  and the continuity of the linear operators $\sqrt{x_i}\partial_{x_i}T(t)$ in $C(Q^d)$ imply that \eqref{limsem} 
is valid for every $u\in C(Q^d)$. 

(2) By \cite[Proposition 2.1.1]{L} and \eqref{eq:s1}  there exists $d_1=d_1(B, \gamma)>0$ such that, for every $\lambda\in\C$, with ${\rm Re}\lambda>\alpha$ and $u\in C(Q^d)$, we have
\[
 ||R(\lambda,\cL^{\gamma,b}) u||_\infty \leq d_1 \frac{||u||_\infty}{|\lambda-\alpha|}.
\]
It follows, for every  $\lambda\in\C$ with  ${\rm Re}\lambda>2\alpha$  and $u\in C(Q^d)$, that   
\[ 
||R(\lambda,\cL^{\gamma,b}) u||_\infty \leq 2d_1 \frac{||u||_\infty}{|\lambda|}, 
\]
i.e., \eqref{eq:r1} is satisfied.
 
Now,  by \eqref{eq:s2} and \eqref{eq:s3} we can derivate under the sign of integral and so we obtain, for every $\eta> \alpha$, $u\in C(Q^d)$, $x\in Q^d$ and $i=1,\ldots,d$, that
\[
{\sqrt{x_i}}\partial_{x_i} \left(\int_0^{+\infty} e^{-\eta t} T(t)udt\right)=\int_0^\infty e^{-\eta t} {\sqrt{x_i}} \partial_{x_i}(T(t) u)dt, 
 \]
 and hence, that
\begin{equation}\label{eq:r4}
\sqrt{x_i}\partial_{x_i}( R(\lambda,\cL^{\gamma,b})u)= \int_0^\infty e^{-\eta t} {\sqrt{x_i}} \partial_{x_i}(T(t) u)dt.
\end{equation} 
By \eqref{eq:s2} and \eqref{eq:s3}  it follows from \eqref{eq:r4} that $\sqrt{x_i}\partial_{x_i}(R(\lambda,\cL^{\gamma,b}) u)\in C(Q^d)$. Moreover,  by applying Lebesgue domination theorem and \eqref{limsem} in  \eqref{eq:r4} we deduce  that
\eqref{limres} is valid.

The proof of \eqref{eq:r2} follows,  with minor changes, as in the proof of \cite[Propositions 2.1(3)]{AM-2}.

(3) It follows as in the proof of Corollary \ref{c.1-gradiente}.
\end{proof}


\section{Main Results}

We are here first concerned with  the following $d$--dimensional  second order elliptic differential operator
\begin{equation}\label{e.opergen-d}
L=\Gamma(x)\sum_{i=1}^d[\gamma_i(x_i)x_i\partial^2_{x_i}+b_i(x)\partial_{x_i}],\quad x\in Q^d=[0,M]^d,
\end{equation}
where $M>0$,   $\Gamma$, $b_i$ and $\gamma_i$, for $i=1,\ldots,d$, are continuous functions on $Q^d$ and on $[0,M]$ respectively. We assume that 

\begin{hypotheses}\label{hypo}The coefficients $\Gamma$, $\gamma_i$ and $b_i$, for $i=1,\ldots,d$, are continuous functions satisfying the following conditions.
\begin{itemize}
\item[\rm (i)] The functions  $\Gamma$ and $\gamma_i$, for $i=1,\ldots,d$, are strictly positive on $Q^d$ and on $[0,M]$ respectively. 
\item[\rm (ii)]  Let $b(x)=(b_1(x),\ldots, b_d(x))$ for $x\in Q^d$. Then $\langle b(x),\nu(x)\rangle \geq 0$ for every $x\in \partial Q^d_0$, where $\partial Q^d_0=\cup_{i=1}^d\{x\in Q^d:\  x_i=0 \}$ and $\nu$ denotes the unit inward normal at $\partial Q^d$.
\end{itemize}
\end{hypotheses}

In the sequel we denote by $L_0$ the operator defined according  to \eqref{e.opergen-d} with $b(x)=0$ and $\Gamma(x)=1$ for $x\in Q^d$. Moreover, for each $i=1,\ldots, d$, we
denote by $S_i\colon C(Q^d)\to C(\partial(Q^d)_i)$ (by $S^i\colon C(Q^d)\to C(\partial(Q^d)^i)$, resp.) the usual restriction map defined by setting $(S_iu)(x)=u(x)$ for $x\in\partial(Q^d)_i$ ($(S^iu)(x)=u(x)$ for $x\in\partial(Q^d)^i$, resp.), where the sets $\partial(Q^d)_i$ and $\partial(Q^d)^i$ are defined as in   \S 3. It is easy to verify that $S_i(C^k(Q^d))\subseteq C^k(\partial(Q^d_0)_i)$  (that $S^i(C^k(Q^d))\subseteq C^k(\partial(Q^d)^i)$, resp.) for every $k\in\N_0$.

\begin{lemma}\label{l.posmin}  Suppose  Hypotheses \ref{hypo} are fulfilled. If  $u\in C^2_{\diamond}(Q^d)$ and $x_0\in Q^d$ is a point where $u$ achieves its minimum, then 
\begin{equation}\label{e.lem-aus}
L_0u(x_0)\geq 0,\quad  \langle b(x_0),\nabla u(x_0)\rangle \geq 0.
\end{equation}
\end{lemma}

\begin{proof} 
To prove \eqref{e.lem-aus} we proceed by induction on the space dimension $d$. Suppose $d=1$. If  $x_0\in (0,M)$, then  \eqref{e.lem-aus}  is clearly satisfied. Otherwise, if $x_0=0$ then we have $L_0u(0)=0$. Moreover, as  $u'(0)$ is necessarily non negative and $b(0)\geq 0$ by Hypothesis \ref{hypo}(ii), we have $b(0)u'(0)\geq 0$. Finally, if $x_0=M$ then $u'(M)=0$. This implies that $u''(M)\geq 0$. Otherwise, $u'(M)=0$ and $u''(M)<0$ ensures that $u'>0$ in $(1-r,1)$ for some $r>0$ and hence, $x_0=M$ cannot be a point of minimum. So, also in this case we have $L_0u(M)\geq 0$ and $b(M)u'(M)=0$.

Next, suppose that \eqref{e.lem-aus} holds for some $d-1\geq 1$. If $x_0\in \stackrel{\circ}{Q^d}$, then     \eqref{e.lem-aus} is clearly satisfied. Otherwise, if $x_0\in \partial Q^d$, then either $x_{0,i}=0$ or $x_{0,i}=M$ for some $i\in\{1,\ldots, d\}$.

Suppose $x_0\in \partial(Q^d)_i$. By Hypothesis \ref{hypo}(ii) we have $b_i(x)\geq 0$ for $x\in \partial(Q^d)_i$. On the other hand, as $x_0\in \partial(Q^d)_i$ is a minimum point for $u$,  $t_0=0$ is a minimum point for the function $\varphi\colon [0,M]\to\R$ defined by $\varphi(t):=u(x_0+te_i)$ for $t\in [0,M]$ and hence, $\partial_{x_i}u(x_0)=\varphi'(0)\geq 0$. So,  $b_i(x_0)\partial_{x_i}u(x_0)\geq 0$.  This implies that
\begin{eqnarray}\label{eq.drift}
\langle b(x_0), \nabla u(x_0)\rangle &=&b_i(x_0)\partial_{x_i}u(x_0)+\sum_{j=1,j\not=i}^d b_j(x_0)\partial_{x_j}u(x_0)\nonumber\\
&\geq &\sum_{j=1,j\not=i}^d b_j(x_0)\partial_{x_j}u(x_0).
\end{eqnarray}
On the other hand, if we set $b^i=(b_j)_{j=1,j\not=i}^d$ then we have
\[
\sum_{j=1,j\not=i}^d b_j(x)\partial_{x_j}u(x)=\sum_{j=1,j\not=i}^d b_j(x)\partial_{x_j}(S_iu)(x)=\langle b^i(x), \nabla (S_iu)(x)\rangle, \quad x\in   \partial(Q^d)_i.
\] 
This, together with \eqref{eq.drift}, implies that 
\begin{equation}\label{e.cond-2}
\langle b(x_0), \nabla u(x_0)\rangle \geq \langle b^i(x_0), \nabla (S_iu)(x_0)\rangle.
\end{equation}
Moreover,  for every $x\in \partial (Q^d)_i$ we have
\begin{equation}\label{e.cond3}
L_0u(x)=\sum_{j=1,j\not=i}^d\gamma_j(x)x_j\partial_{x_j}^2u(x)=(L_0)|_{\partial (Q^d)_i}(S_iu)(x).
\end{equation}
Next, we observe that $S_iu\in C_{\diamond}^2(\partial (Q^d)_i)$. Also, the operator $(L_0)|_{\partial (Q^d)_i}$ and the vector $b^i$ satisfy the same hypotheses of this lemma on $\partial (Q^d)_i$ and $x_0$ is a minimum point for $S_iu$ on $\partial (Q^d)_i$. So,  we may apply the inductive hypothesis to conclude that
\[
(L_0)|_{\partial (Q^d)_i}(S_iu)(x_0)\geq 0,\quad \langle b^i(x_0), \nabla  (S_iu)(x_0)\rangle\geq 0.
\]
By \eqref{e.cond-2} and \eqref{e.cond3} this implies that 
\[
L_0u(x_0)\geq 0, \quad \langle b(x_0), \nabla u(x_0)\rangle \geq 0,
\]
and hence, the proof of the first part is complete.

Suppose that $x_0\in \partial(Q^d)^i$. Then $\partial_{x_i}u(x_0)=0$ and hence, $\partial_{x_i}^2u(x_0)\geq 0$ as it is easy to prove argumenting as above when $d=1$.  It follows that 
\begin{equation}\label{e.cond-4}
\langle b(x_0),\nabla u(x_0)\rangle =\sum_{j=1, j\not =i}^db_j(x_0)\partial_{x_j}u(x_0)=\langle b^i(x_0),\nabla (S^iu)(x_0)\rangle,
\end{equation}
and
\begin{equation}\label{e.cond-5}
L_0u(x_0)=\sum_{j=1}^d \gamma_j(x_0)x_j\partial_{x_j}^2u(x_0)\geq \sum_{j=1,j\not=i}^d \gamma_j(x_0)x_j\partial_{x_j}^2u(x_0)=(L_0)|_{\partial (Q^d)^i}(S^iu)(x_0).
\end{equation}
Now, we observe that $S^iu\in C_{\diamond}^2(\partial (Q^d)^i)$. Also, the operator $(L_0)|_{\partial (Q^d)^i}$ and the vector $b^i$ satisfy the same hypotheses of this lemma on $\partial (Q^d)^i$ and $x_0$ is a minimum point for $S^iu$ on $\partial (Q^d)^i$. So,  we may apply the inductive hypothesis to conclude that
\[
(L_0)|_{\partial (Q^d)^i}(S^iu)(x_0)\geq 0,\quad \langle b^i(x_0), \nabla  (S^iu)(x_0)\rangle\geq 0.
\]
By \eqref{e.cond-4} and \eqref{e.cond-5} this implies that 
\[
L_0u(x_0)\geq 0, \quad \langle b(x_0), \nabla u(x_0)\rangle \geq 0.
\]
Hence, the proof is complete.
\end{proof}

As an immediate consequence of Lemma \ref{l.posmin} we obtain that the operator $(L, C^2_{\diamond}(Q^d))$ satisfies the following \textit{maximum principle}.

\begin{corollary}\label{c.dissoper}  Suppose Hypotheses \ref{hypo}  are fulfilled. If $u\in C^2_{\diamond}(Q^d)$ satisfies the following inequality
\begin{equation}\label{e.posmin}
\lambda u(x)-Lu(x)\geq 0,\quad x\in Q^d,
\end{equation}
for some $\lambda>0$, then $u(x)\geq 0$ for every $x\in Q^d$.

In particular, the operator $(L,  C^2_{\diamond}(Q^d))$ is dissipative and closable in $C(Q^d)$, with dissipative closure.
\end{corollary}

\begin{proof}
Suppose that $u\in C^2_\diamond(Q^d)$ is such that \eqref{e.posmin}  holds but, $u(\overline{x})<0$ for some $\overline{x}\in Q^d$. Let  $u(x_0)=\min_{x\in Q^d}u(x)$. Then $u(x_0)\leq u(\overline{x})<0$. So, by \eqref{e.lem-aus} we have $L_0u(x_0)\geq 0$ and $\langle b(x_0), \nabla u(x_0)\rangle\geq 0$. This implies that $Lu(x_0)=\Gamma(x_0)[L_0u(x_0)+\langle b(x_0), \nabla u(x_0)\rangle]\geq 0$ and hence,
\[
\lambda u(x_0)-Lu(x_0)\leq \lambda u(x_0)<0;
\]
this is a contradiction. So, the first part of corollary follows.

Fix  $u\in C^2_{\diamond}(Q^d)$. By Lemma \ref{l.posmin} we may assume that $0<u(x_0)=\|u\|_\infty$ for some $x_0\in Q^d$. So, again by Lemma \ref{l.posmin} we have
\[
\|\mu u\|_{\infty}=\mu u(x_0)\leq \mu u(x_0)-Lu(x_0)\leq \|(\mu -L)u\|_\infty
\]
for every $\mu>0$. This means that the operator $(L,C^2_\diamond(Q^d))$ is dissipative. 

On the other hand, $C^2_\diamond(Q^d)$ is dense in $C(Q^d)$.
 As $(L,C^2_\diamond(Q^d))$ is dissipative with dense domain,  it follows that the operator $(L, C^2_\diamond(Q^d))$ is dissipative and closable in $C(Q^d)$, with dissipative closure.
\end{proof}

\begin{remark}\label{r.firstorder}\rm Let $(b_i)_{i=1}^d\su C(Q^d)$ with $b_i(x)=0$ for every $x\in \partial(Q^d)_i$ and $i=1,\ldots, d$. Let $B:=\sum_{i=1}^db_i(x)\partial_{x_i}$ be the first order differential operator with domain $C^1_\diamond(Q^d):=\cap_{i=1}^d\{u\in C^1(Q^d): \forall x\in \partial(Q^d)_i\ \partial_{x_i}u(x)=0\}$. Then $Bu(x_0)=0$ whenever $u\in C^1_\diamond(Q^d)$ and  $x_0\in Q^d$  is a point in which $u$  achieves its minimum. The proof of this fact is along the lines of the one of Lemma \ref{l.posmin} but \textit{simpler}. So, via similar arguments to Corollary \ref{c.dissoper}, we may conclude  that the operator $(B, C^1_\diamond(Q^d))$ is dissipative  and closable in $C(Q^d)$ with dissipative closure.
\end{remark}
To show the first main result of this paper we need a further hypothesis.

\begin{hypothesis}\label{hypo-2} The coefficients $b_i$, for $i=1,\ldots,d$, satisfy the following condition.
\begin{itemize}
\item[\rm (iii)] There exists $\delta>0$ and $C>0$ such that, for every $i=1,\ldots,d$ and $x,\ x'\in Q^d$ with $x_i<\delta$ and $x_i'=0$, we have
\begin{equation}\label{e.ass-2}
|b_i(x)-b_i(x')|\leq C\sqrt{x_i},
\end{equation}
\end{itemize}
\end{hypothesis}

\begin{remark}\label{r.hypo-2}\rm (a) Hypothesis \ref{hypo-2} implies, for each $i=1,\ldots, d$,  that the function $b_i$ is constant on the face $\partial(Q^d)_i$, namely $b_i(x)=b_i(0)$ for every $x\in \partial(Q^d)_i$.

(b) Under Hypothesis \ref{hypo-2}, the continuity of each $b_i$ on $Q^d$ yields that there exists a new constant $C'>0$ such that
\[
|b_i(x)-b_i(x')|\leq C'\sqrt{x_i},
\]
for every $x,\ x'\in Q^d$ with $x'_i=0$ and $i=1,\ldots, d$.
\end{remark}

\begin{theorem}\label{l.dominio-anal} Suppose that Hypotheses \ref{hypo} and \ref{hypo-2}, with $\Gamma\equiv 1$ on $Q^d$, are valid. Set $b=(b_1(0), b_2(0),\ldots, b_d(0))$, $\gamma=(\gamma_1,\gamma_2, \ldots,\gamma_d)$ and  denote by $(\cL^{\gamma,b}, D(\cL^{\gamma, b}))$ the differential operator with constant drift term defined according to  \S 3. Then the closure $(\cL, D(\cL))$ of the operator $(L,C^2_\diamond(Q^d))$ satisfies the following properties.
\begin{itemize}
\item[\rm (1)] $D(\cL^{\gamma,b})=D(\cL)$.
\item[\rm (2)] $(\cL, D(\cL))$ generates a bounded analytic  $C_0$--semigroup $(T(t))_{t\geq 0}$ of positive contractions in $C(Q^d)$ and of angle $\pi/2$. The semigroup $(T(t))_{t\geq 0}$  is compact. 
\end{itemize}
\end{theorem}

\begin{proof} (1) We first show that $D(\cL^{\gamma,b})\su D(\cL)$. The converse inclusion follows after proving part (2) of this theorem.

Fix any  $f\in D(\cL^{\gamma,b})$. By Proposition \ref{p.-dcore} there exists $(f_n)_n\su C^2_\diamond(Q^d)$ such that $f_n\to f$ and $\cL^{\gamma,b}f_n\to \cL^{\gamma,b}f$ in $C(Q^d)$ as $n\to\infty$. 
Now, we observe that by Proposition \ref{p.dimd-grad}(3) we have
\[
\sqrt{x_i}|\partial_{x_i}(f_n(x)-f_m(x))| \leq a\|f_n-f_m\|_\infty+b\|\cL^{\gamma,b}(f_n-f_m)\|_\infty,\ i=1,\ldots, d, \ n,\ m\in\N,
\]
for some positive constants $a,\ b$.  So, for each $i\in\{1,\ldots, d\}$, it follows that $(\sqrt{x_i}\partial_{x_i}f_n)_n$ is a Cauchy sequence in $C(Q^d)$. On the other hand,  Hypothesis \ref{hypo-2} and   Remark \ref{r.hypo-2}(b) yield that
\[
|b_i(x)-b_i(0)||\partial_{x_i}(f_n(x)-f_m(x))|\leq C\sqrt{x_i}|\partial_{x_i}(f_n(x)-f_m(x)|, \ i=1,\ldots, d, \ n,\ m\in\N,
\]
and so $((b_i(x)-b_i(0))\partial_{x_i}f_n)_n$ is also a Cauchy sequence in $C(Q^d)$. Thus, for each $i=1,\ldots, d$, $(b_i(x)-b_i(0))\partial_{x_i}f_n\to g_i$ in $C(Q^d)$ as $n\to\infty$. Since
\begin{equation}\label{eq.teorema-1}
\cL f_n=\cL^{\gamma,b}f_n+\sum_{i=1}^d (b_i(x)-b_i(0))\partial_{x_i}f_n,\quad n\in\N, 
\end{equation}
we deduce that $\cL f_n\to \cL^{\gamma,b}f+\sum_{i=1}^d g_i=:g$ in $C(Q^d)$ as $n\to\infty$. After observing that $(f_n)_n\su C^2_\diamond(Q^d)\su D(\cL)$ with $f_n\to f$ in $C(Q^d)$ as $n\to\infty$ and that $(\cL,D(\cL))$ is a closed operator, we conclude that $f\in D(\cL)$ and $\cL f=g$. So, the proof is complete.

(2) Set $B:=\sum_{i=1}^d (b_i(x)-b_i(0))\partial_{x_i}$. Since the coefficients $c_i(x):=b_i(x)-b_i(0)$, for $i=1,\ldots, d$, satisfy the hypotheses in Remark \ref{r.firstorder}, i.e., $c_i(x)=0$  for $x\in \partial(Q^d)_i$ and $i=1,\ldots, d$, the operator $(B, C^1_\diamond(Q^d))$ is dissipative and closable in $C(Q^d)$ with dissipative closure. 
We claim that  the closure $(\cB, D(\cB))$ of the operator  $(\cB, C^1_\diamond (Q^d))$ is  $\cL^{\gamma, b}$--bounded with $\cL^{\gamma, b}$--bound $a_0=0$.
To show this claim, we first observe that $C^2_\diamond(Q^d)\su C^1_\diamond(Q^d)$. On the other hand,  by Hypothesis \ref{hypo-2} and   Remark \ref{r.hypo-2}(b) there exists $C'>0$ such that
\[
|b_i(x)-b_i(0)||\partial_{x_i}u(x)|\leq C'\sqrt{x_i}|\partial_{x_i}u(x)|, \ i=1,\ldots, d, \ u\in C^2_\diamond (Q^d).
\]
So, by applying  Proposition \ref{p.dimd-grad}(3) we obtain that, 
 there exist $C'',D', \ov{\ve}>0$ such that,   for every $0<\ve<\ov{\ve}$, we have
\begin{equation}\label{e.bound}
\|Bu\|_\infty\leq D'\ve \|\cL^{\gamma,b}u\|_\infty+\frac{C''}{\ve}\|u\|_\infty,\quad  u\in C^2_\diamond (Q^d).
\end{equation}
Since $C^2_\diamond (Q^d)$ is a core for $(\cL^{\gamma, b},D(\cL^{\gamma, b}))$ (see, Proposition \ref{p.-dcore} or Corollary \ref{c.dissoper}), this inequality remains valid for every  $u\in D(\cL^{\gamma, b})$. 
Indeed, if $u\in  D(\cL^{\gamma, b})$, there exists a sequence $(f_n)_n\su C^2_\diamond(Q^d)$ such that $f_n\to u$ and $\cL^{\gamma, b}f_n\to \cL^{\gamma, b}u$ as $n\to\infty$. So, $(f_n)_n$ and $(\cL^{\gamma, b}f_n)_n$ are Cauchy sequences in $C(Q^d)$. By \eqref{e.bound} it follows that $(Bf_n)_n$ is also a Cauchy sequence in $C(Q^d)$ and so it converges to some $g$ in $C(Q^d)$.  But $C^2_\diamond(Q^d)\subseteq C^1_\diamond (Q^d)$ and  the operator $(\cB, D(\cB))$ is closed and hence, $u\in D(\cB)$ and $\cB u=g$ (we point out that this fact proves also that $D(\cL^{\gamma, b})\su D(\cB)$).  Finally, by replacing $u$ with  $f_n$  in \eqref{e.bound} and passing to the limit for $n\to\infty$, it follows that  \eqref{e.bound} remains valid for such a fixed $u$ in   $D(\cL^{\gamma, b})$. So, for every $0<\ve<\ov{\ve}$, we have
\begin{equation}\label{e.bound-1}
\|\cB u\|_\infty\leq D'\ve \|\cL^{\gamma,b}u\|_\infty+\frac{C''}{\ve}\|u\|_\infty,\quad  u\in D(\cL^{\gamma,b}).
\end{equation}
Inequality \eqref{e.bound-1} ensures  that  the operator $(\cB, D(\cB))$ is  $\cL^{\gamma, b}$--bounded with $\cL^{\gamma, b}$--bound $a_0=0$. Thus, the operator $(\cL^{\gamma, b}+\cB, D(\cL^{\gamma, b}))$ is closed and generates a bounded analytic $C_0$--semigroup $(T(t))_{t\geq 0}$ of  contractions in $C(Q^d)$ and angle $\pi/2$, \cite[Chap. III, \S2, Lemma 2.4, Theorems 2.7 \& 2.10]{EN}. The semigroup $(T(t))_{t\geq 0}$ is also compact as $(T(t))_{t\geq 0}$ is norm--continuous, being it analytic, and the operator $(\cL^{\gamma, b}+\cB, D(\cL^{\gamma, b}))$ has compact resolvent (see the proof of \cite[Chap. III, \S2, Lemma 2.10]{EN}).

Now, we observe that $\cL=\cL^{\gamma, b}+\cB$ and $D(\cL)=D(\cL^{\gamma, b})$. Indeed,  by \eqref{e.bound-1} we obtain, for every  $0<\ve<\ov{\ve}$ and $u\in D(\cL^{\gamma,b})$, that
\begin{equation}\label{e.bound-11}
\|\cL^{\gamma,b}u\|_\infty=\|(\cL^{\gamma,b}+\cB )u-\cB u\|_\infty\leq \|(\cL^{\gamma,b}+\cB )u\|_\infty+\ve D'\|\cL^{\gamma,b}u\|_\infty+\frac{C''}{\ve}\|u\|_\infty.
\end{equation}
Since $D(\cL^{\gamma,b})\su D(\cL)$ by part (1) and hence, $\cL=\cL^{\gamma,b}+\cB$ on $D(\cL^{\gamma,b})$, it follows that
\begin{equation}\label{e.bound-2}
\| \cL^{\gamma,b}u\|_\infty\leq \frac{1}{1-D'\ve_0}\|\cL u\|_\infty+\frac{C''}{\ve_0 (1-D'\ve_0)}\|u\|_\infty, \quad u\in D(\cL^{\gamma,b}),
\end{equation}
after having taken in \eqref{e.bound-11} $\ve_0$ small enough to have $D'\ve_0<1$. Using again the facts that $C^2_\diamond (Q^d)(\su D(\cL^{\gamma,b}))$ is a core for $(\cL, D(\cL))$ by Corollary \ref{c.dissoper} and that the operator $(\cL^{\gamma,b}, D(\cL^{\gamma,b}))$ is closed, we deduce from \eqref{e.bound-2} that $D(\cL)\su D(\cL^{\gamma,b})$. This completes the proof of part (1) of this theorem and ensures that $\cL=\cL^{\gamma, b}+\cB$.

As $D(\cL)=D(\cL^{\gamma, b})$ and hence $\cL=\cL^{\gamma, b}+\cB$,  the proof of part (2) is now complete. We point out  that the positivity of the semigroup $(T(t))_{t\geq 0}$ follows from Corollary \ref{c.dissoper}.
\end{proof}

Thanks to the proof of Theorem \ref{l.dominio-anal} we are able to show that the operator $(\cL,D(\cL))$ shares further properties of the operator $(\cL^{\gamma,b}, D(\cL^{\gamma, b}))$. Precisely, we   prove 
  estimates  for the norm of the resolvent  operators of $(\cL, D(\cL))$ and of their gradient with   constants which don't depend on the function $b$.

\begin{prop}\label{p.ddim-gradiente} Suppose that Hypotheses \ref{hypo} and \ref{hypo-2}, with $\Gamma\equiv 1$ on $Q^d$, are valid and let  $B\in\R$ with $B\geq \max_{i=1}^d\|b_i\|_\infty$, $\gamma:=(\gamma_1,\ldots,\gamma_d)$. Then the closure $(\cL,D(\cL))$ of the operator $(L, C^2_\diamond(Q^d))$ satisfies the following properties.
\begin{itemize}
\item[\rm (1)] There exist $K, \alpha, \overline t>0$ depending on $B$ and on $\gamma$ such that, for every $u\in C(Q^d)$ and $i=1,\ldots,d$, we have
\begin{eqnarray}
& & ||t\cL^{\gamma,b} T(t)|| \leq Ke^{\alpha t},\quad t\geq 0.\\
& &||\sqrt{x_i}\partial_{x_i}(T(t)u)||_\infty\leq \frac{K e^{\alpha t}}{\sqrt t}||u||_\infty,        \quad 0< t<\overline t.\\
& &||\sqrt{x_i}\partial_{x_i}(T(t)u)||_\infty\leq K e^{\alpha t}||u||_\infty, \qquad  t\geq \overline t.
\end{eqnarray}
Moreover, for every $i\in \{1,\ldots,d\}$ and $u\in C(Q^d)$,  $\sqrt{x_i}\partial_{x_i}(T(t)u)\in C(Q^d)$ and 
\begin{equation}\label{limsem-varcoeff}
 \lim_{x_i\to 0^+} \sup_{x_j\in [0,M], j\in \{1,\ldots,d\}\setminus\{i\}}\sqrt{x_i}\partial_{x_i}(T(t)u)=0.
\end{equation}
\item[\rm (2)] There exist  $d_1, d_2, R>0$ depending on $B$ and on $\gamma$  such that, for every $\lambda\in \C$ with  ${\rm Re}\lambda>R$ and for every $f\in C(Q^d)$, we have
\begin{eqnarray}
\label{e.varcoeff} & &||R(\lambda,\cL^{\gamma,b}) u||_\infty \leq d_1 \frac{||u||_\infty}{|\lambda|},\\
\label{e.varcoeff-1} & & ||\sqrt{x_i}\partial_{x_i}(R(\lambda,\cL^{\gamma,b}) u)||_\infty \leq d_2 \frac{||u||_\infty}{\sqrt{|\lambda|}}.
 \end{eqnarray}
Moreover, for every $i\in \{1,\ldots,d\}$ and  $u\in C(Q^d)$, $\sqrt{x_i}\partial_{x_i}(R(\lambda,\cL^{\gamma,b}) u)\in C(Q^d)$ and that
\begin{equation}\label{limres-varcoeff} 
\lim_{x_i\to 0^+} \sup_{x_j\in [0,M],j\in \{1,\ldots,d\}\setminus\{i\}}\sqrt{x_i}\partial_{x_i}(R(\lambda,\cL^{\gamma,b}) u)(x)=0.
\end{equation}
\item[\rm (3)] There exist  $C, D, \overline \varepsilon>0$  depending on $B$ and on $\gamma$ such that,  for every $0<\varepsilon<\overline\varepsilon$,  $i=1,\ldots, d$  and    
$u\in D(\cL^{\gamma,b})$, we have 
\[ 
\|\sqrt{x_i} \partial_{x_i}u\|_\infty\leq \frac{C} {\varepsilon}  \|u\|_\infty+ D\varepsilon \|\cL^{\gamma,b}u\|_\infty.
 \]
\end{itemize}
\end{prop}

\begin{proof} Set $b=(b_1(0), b_2(0),\ldots, b_d(0))$ and  denote by $(\cL^{\gamma,b}, D(\cL^{\gamma, b}))$ and by $(\cB,D(\cB))$ the differential operators already considered in the proof of Theorem \ref{l.dominio-anal}.


We begin  showing first  part (3). By Proposition \ref{p.dimd-grad}(3) there exist $C,D,\ov{\ve}>0$ depending on $B$ and $\gamma$ such that, for every $0<\ve<\ov{\ve}$, $i=1,\ldots, d$ and $u\in D(\cL)(=D(\cL^{\gamma,b}))$, we have
\begin{equation}\label{e.stima-1}
\|\sqrt{x_i}\partial_{x_i}u\|_\infty\leq \frac{C}{\ve}\|u\|_\infty+D\ve \|\cL^{\gamma,b}u\|_\infty.
\end{equation}
If we take in \eqref{e.bound-2} $D\ve_0<1/2$, then from \eqref{e.stima-1} and \eqref{e.bound-2} it follows,  for every $0<\ve<\ov{\ve}$, $i=1,\ldots, d$ and $u\in D(\cL)$, that
\begin{eqnarray}\label{e.stima-2}
\|\sqrt{x_i}\partial_{x_i}u\|_\infty &\leq & \frac{C}{\ve}\|u\|_\infty+D\ve \|\cL^{\gamma,b}u\|_\infty\nonumber\\
&\leq & \frac{C}{\ve}\|u\|_\infty+D\ve \left(2\|\cL u\|_\infty+\frac{2C''}{\ve_0 }\|u\|_\infty\right)\nonumber\\
&=&\frac{C+2C''\ve_0^{-1}}{\ve}\|u\|_\infty+2D\ve \|\cL u\|_\infty .
\end{eqnarray}
This completes the proof of part (3).

(2) By Proposition \ref{p.dimd-grad}(2) and by \eqref{e.bound} there exist $d_1, R>0$ and $C'',D', \ov{\ve}>0$ such that, for every $0<\ve<\ov{\ve}$, $\lambda\in\C$ with ${\rm Re}\lambda>R$ and $u\in C(Q^d)$, we have
\begin{eqnarray}\label{e.stima-3}
\|\cB R(\lambda, \cL^{\gamma, b})u\|_\infty &\leq & D'\ve \|\cL^{\gamma, b}R(\lambda,\cL^{\gamma,b})u\|_\infty+\frac{C''}{\ve}\|R(\lambda,\cL^{\gamma,b})u\|_\infty\nonumber\\
&\leq & D'\ve \|\lambda  R(\lambda,\cL^{\gamma,b})u-u\|_\infty+\frac{C''}{\ve}\|R(\lambda,\cL^{\gamma,b})u\|_\infty\nonumber\\
&\leq & D'\ve(d_1+1)\|u\|_\infty+\frac{C''}{\ve}\frac{d_1}{|\lambda|}\|u\|_\infty.
\end{eqnarray}
Now, fix $\ve_0\in (0,\ov{\ve})$ such that $D'\ve(d_1+1)<1/4$ and  choose $R'>R$ such that $\frac{C''}{\ve_0}\frac{d_1}{|\lambda|}<1/4$ for all $\lambda\in \C$ with ${\rm Re}\lambda>R'$. So, from \eqref{e.stima-3} it follows, for every $\lambda\in \C$ with ${\rm Re}\lambda>R'$, that  $\|\cB R(\lambda, \cL^{\gamma, b})\|<1/2$. Hence, for every $\lambda\in \C$ with ${\rm Re}\lambda>R'$, the continuous linear operator $\cB R(\lambda, \cL^{\gamma, b})$ is invertible in $\cL(C(Q^d))$ with inverse given by 
\begin{equation}\label{e.inversa}
R(\lambda, \cL^{\gamma,b}+\cB)=R(\lambda, \cL^{\gamma, b})\sum_{n=0}^{\infty}(\cB R(\lambda, \cL^{\gamma, b})),
\end{equation}
(see \cite[Chap. III, \S 2, Lemma 2.5]{EN}). Since $\cL=\cL^{\gamma, b}+\cB$ with $D(\cL)=D(\cL^{\gamma, b})$ (see the proof of part (2) of Theorem \ref{l.dominio-anal}), by \eqref{e.inversa} and Proposition \ref{p.dimd-grad}(2) we obtain, for every $\lambda\in \C$ with ${\rm Re}\lambda>R'$ and $f\in C(Q^d)$, that
\begin{equation}\label{e.stima-4}
\|R(\lambda, \cL)f\|\leq 2\|R(\lambda, \cL^{\gamma, b})\|\|f\|_\infty\leq \frac{2d_1}{|\lambda|} \|f\|_\infty,
\end{equation}
and that
\begin{equation}\label{e.stima-5}
\|\sqrt{x_i}\partial_{x_i}(R(\lambda, \cL)f)\|\leq \frac{2d_2}{\sqrt{|\lambda|}}\|f\|_\infty, \quad i=1,\ldots, d.
\end{equation}
Moreover, Proposition \ref{p.dimd-grad}(2) and formula \eqref{e.inversa} also imply, for every  $i=1,\ldots, d$,  that $\sqrt{x_i}\partial_{x_i}(R(\lambda, \cL)f)\in C(Q^d)$  and that $\lim_{x_i\to 0^+}\sup_{x_j\in [0,M], j\not =i}\sqrt{x_i}\partial_{x_i}(R(\lambda, \cL)f)(x)=0$.

(1) Part (1) follows as in the proof of Proposition \ref{p.1-gradiente}(2) taking into account that the operator $(\cL,D(\cL))$ satisfies part (2) of this proposition and generates a contractive analytic  $C_0$--semigroup in $C(Q^d)$.
\end{proof}

Now, we can show that the operator $(\cL, D(\cL))$ with $\Gamma$ any strictly positive continuous function on $Q^d$ also generates an analytic compact $C_0$--semigroup $(T(t))_{t\geq 0}$ of positive contractions in $C(Q^d)$. In order to prove this, we state the following lemma whose proof is straightforward.

\begin{lemma}\label{l.part} For each $n\in\N$ and $i\in \{1,\ldots, n-1\}$ set $I^i_n=\left[\frac{i-1}{n},\frac{i+1}{n}\right]$ and  let $\{\varphi^i_n\}_{i=1}^{n-1}\subset C^\infty(\R)$ such that $\sum_{i=1}^{n-1}(\varphi^i_n)^2\equiv 1$ on $[0,1]$, 
${\rm supp} (\varphi^i_n)\subset \left[\frac{i-1}{n},\frac{i+1}{n}\right]$ for $i=2,\ldots, n-2$, ${\rm supp} (\varphi^1_n)\subset \left]-\infty, \frac{2}{n}\right]$ and ${\rm supp} (\varphi^{n-1}_n)\subset \left[\frac{n-2}{n}, \infty\right[$. 

For each $x\in [0,1]^d$ and  $\mathbf{i}=(i_1,\ldots, i_d)\in J^d_n=\{1,\ldots, n-1\}^d$ set 
\begin{equation}\label{e.fun}
\Phi^{\mathbf{i}}_n(x)=\prod_{h=1}^d\varphi^{i_h}_n(x_h)
\end{equation}
 Then ${\rm supp}(\Phi^\mathbf{i}_n)\subset \prod_{h=1, i_h\not\in\{1,n-1\}}^d  I^{i_h}_n\times \prod_{h=1, i_h=1}^d\left]-\infty, \frac{2}{n}\right]\times \prod_{i_h, i_h=n-1}^d\left[\frac{n-2}{n},\infty\right[$ (in the suitable order) and $\sum_{\mathbf{i}\in J^d_n }(\Phi^\mathbf{i}_n)^2\equiv 1$ on $[0,1]^d$ for every $n\in\N$ and $\mathbf{i}\in J^d_n$. Moreover, if  $v=\sum_{\mathbf{i}\in J^d_n}\Phi^\mathbf{i}_nv_\mathbf{i}$ for some $\{v_\mathbf{i}\}_{\mathbf{i}\in J^d_n}\subset C(Q^d)$, then there exists $J\subset J^d_n$ such that $|J|\leq 3^d$ and $v=\sum_{\mathbf{i}\in J}\Phi^\mathbf{i}_nv_\mathbf{i}$ and hence, $\|v\|_\infty\leq 3^d \sup_{\mathbf{i}\in J^d_n}\|\Phi^\mathbf{i}_nv_\mathbf{i}\|_\infty$.
\end{lemma}

\begin{theorem}\label{t.main1} Under Hypotheses \ref{hypo} and \ref{hypo-2}, the operator $(\cL, D(\cL))$  generates an analytic compact $C_0$--semigroup $(T(t))_{t\geq 0}$ of positive contractions in $C(Q^d)$.  Moreover, all the estimates in  Proposition \ref{p.ddim-gradiente} hold for  the operator $(\cL, D(\cL))$.
\end{theorem}

\begin{proof} Without loss of generality we can suppose $M=1$, i.e., $Q^d=[0,1]^d$.

Denote by $(\cL_1, D(\cL_1))$ the closure of the operator  defined according to \eqref{e.opergen-d} with $\Gamma\equiv 1$ on $Q^d$. Then by  Theorem \ref{l.dominio-anal} (see also Corollary \ref{c.dissoper}) $(\cL_1, D(\cL_1))$ generates a (bounded analytic compact) $C_0$--semigroup of positive contractions in $C(Q^d)$ (and of angle $\pi/2$). Since $\cL=\Gamma \cL_1$ we can apply \cite[Theorem 12]{Do} to conclude that  $(\cL,D(\cL_1))$ generates a  $C_0$--semigroup $(T(t))_{t\geq 0}$ of positive contractions in $C(Q^d)$. But $D(\cL)=D(\cL_1)$ and hence, $(\cL,D(\cL))$ generates a  $C_0$--semigroup $(T(t))_{t\geq 0}$ of positive contractions in $C(Q^d)$. The identities  $D(\cL)=D(\cL_1)$ and $\cL=\Gamma \cL_1$ follow from the facts that  $\Gamma$ is a strictly positive continuous function on $Q^d$,  the operators $(\cL,D(\cL_1))$and $(\cL, D(\cL))$ are closed and that $C^2_\diamond(Q^d)$ is a core for both the operators $(\cL,D(\cL_1))$ and $(\cL, D(\cL))$.

We claim that the semigroup $(T(t))_{t\geq 0}$ is analytic. To show this thanks to Theorem \ref{l.dominio-anal} and Proposition \ref{p.ddim-gradiente} we can proceed as in the proof of Propositions 2.6 and 2.7 in \cite{AM-2} and so we indicate here only the main changes. 

In the sequel we follow the notation introduced in Lemma \ref{l.part}. 

For each $n\in \N$ and $\mathbf{i}\in J^d_n$ set $I^{\mathbf{i}}_n=\prod_{h=1}^dI^{i_h}_n$, fix $V_n^{\mathbf{i}}\in I_n^\mathbf{i}$ and define $\Gamma_n^\mathbf{i}=\Gamma(V_n^{\mathbf{i}})$.  

Fix $n\in\N$. By Proposition \ref{p.ddim-gradiente}(2),  there exists $R>0$ depending on $\max_{i=1,\dots,d}||b_i||_\infty$ and on $\gamma_1,\dots\gamma_d, \Gamma$ such that  for every $\lambda\in\C$ with ${\rm Re}\lambda>\frac{R}{\Gamma_0}$ ($\Gamma_0=\min_{x\in Q^d}\Gamma(x)>0$), we can consider the operators defined by
\[
R_n^\mathbf{i}(\lambda)=(\lambda-\Gamma_n^\mathbf{i}\cL_1)^{-1}
\] 
which satisfy
\begin{equation}\label{e:ga1}
\|R_n^\mathbf{i}(\lambda)\|=(\Gamma_n^\mathbf{i})^{-1}\left\|R\left(\frac{\lambda}{\Gamma_n^\mathbf{i}},\cL_1\right)\right\|\leq \frac{d_1}{\Gamma_0}\frac{1}{|\lambda|}.
\end{equation}
On the other hand,  for every $\lambda\in\C$ with ${\rm Re}\lambda>\frac{R}{\Gamma_0}$, the following equality holds
\[
\cL_1R_n^\mathbf{i}(\lambda)= (\Gamma_n^\mathbf{i})^{-1}(-I+\lambda R_n^\mathbf{i}(\lambda))
\]
and hence, via \eqref{e:ga1} we obtain, for every $\lambda\in\C$ with ${\rm Re}\lambda>\frac{R}{\Gamma_0}$, that
\begin{equation}\label{e:ga2}
\|\cL_1R_n^\mathbf{i}(\lambda)\|\leq (\Gamma_n^\mathbf{i})^{-1}\left(1+\frac{d_1}{\Gamma_0}\right)\leq \frac{\Gamma_0+d_1}{\Gamma_0^2}.
\end{equation}
We now 
consider the continuous functions $\{\Phi^{\mathbf{i}}_n\}_{\mathbf{i}\in J^d_n}$ defined  according to \eqref{e.fun} of the above  Lemma \ref{l.part} and define the the approximate resolvents of $(\cL, D(\cL))$ given by 
\[
S_n(\lambda)u=\sum_{\mathbf{i}\in J^d_n}\Phi_n^\mathbf{i}R^\mathbf{i}_n(\lambda)(\Phi_n^\mathbf{i}u)
\] 
for every $\lambda\in \C$ with ${\rm Re}\lambda>\frac{R}{\Gamma_0}$ and $u\in  C(Q^d)$. So, by Lemma \ref{l.part} and \eqref{e:ga1} we obtain, for every $\lambda\in\C$ with ${\rm Re}\lambda>\frac{R}{\Gamma_0}$, that
\begin{equation}\label{e:ga3}
\|S_n(\lambda)\|\leq \frac{3^dd_1}{\Gamma_0}\frac{1}{|\lambda|}.
\end{equation}
Moreover,  for every $\lambda\in\C$ with ${\rm Re}\lambda>\frac{R}{\Gamma_0}$ and $u\in C(Q^d)$, we have
\begin{eqnarray}\label{e:c0}
& & (\lambda-\cL)S_n(\lambda)u = u+\sum_{\mathbf{i}\in J^d_n}(\Gamma_n^\mathbf{i}-\Gamma)\Phi_n^\mathbf{i}\cL_1
(R^\mathbf{i}_n(\lambda)(\Phi_n^\mathbf{i}u))\nonumber\\
& & \qquad -\sum_{\mathbf{i}\in J^d_n}\cL(\Phi_n^\mathbf{i})R^\mathbf{i}_n(\lambda)(\Phi_n^\mathbf{i}u)-2\Gamma\sum_{\mathbf{i}\in J^d_n}\sum_{h=1}^d\gamma_h(x_h)x_h\partial_{x_h}(R^\mathbf{i}_n(\lambda)(\Phi_n^\mathbf{i}u))\partial_{x_h}(\Phi_n^\mathbf{i})\nonumber\\
& &\qquad  =:(I+C_1^n(\lambda)+C_2^n(\lambda)+C_3^n(\lambda))u.
\end{eqnarray}
We now fix $\ov{n}$ such that $\max_{x\in I_{\ov{n}}^\mathbf{i}}|\Gamma(x)-\Gamma_{\ov{n}}^\mathbf{i}|<\ve_0:=\frac{\Gamma_0^2}{4.3^d(\Gamma_0+d_1)}$ for all $\mathbf{i}\in J^d_{\ov{n}}$. Then, from \eqref{e:ga1}, \eqref{e:ga2}, Proposition \ref{p.ddim-gradiente}(2) and Lemma \ref{l.part} it follows, for every $\lambda\in \C$ with ${\rm Re}\lambda>\frac{R}{\Gamma_0}$ and $u\in C(Q^d)$,  that 
\begin{equation}\label{eq:c1}
\|C_1^{\ov{n}}(\lambda)u\|_\infty\leq 3^d\ve_0\max_{\mathbf{i}\in J^d_{\ov{n}}}\|\Phi_{\ov{n}}^\mathbf{i}\cL_1
(R^\mathbf{i}_{\ov{n}}(\lambda)(\Phi_{\ov{n}}^\mathbf{i}u))\|_\infty\leq 3^d\ve_0\frac{\Gamma_0+d_1}{\Gamma_0^2}\|u\|_\infty<\frac{1}{4}\|u\|_\infty,
\end{equation}
\begin{equation}\label{eq:c2}
\|C_2^{\ov{n}}(\lambda)u\|_\infty\leq |J^d_{\ov{n}}|\max_{\mathbf{i}\in J^d_{\ov{n}}}\|\cL(\Phi_{\ov{n}}^\mathbf{i})\|_\infty\|R^\mathbf{i}_{\ov{n}}(\lambda)(\Phi_{\ov{n}}^\mathbf{i}u)\|_\infty\leq |J^d_{\ov{n}}|K_1\frac{d_1}{\Gamma_0}\frac{1}{|\lambda|}\|u\|_\infty,
\end{equation}
\begin{eqnarray}\label{eq:c3}
\|C_3^{\ov{n}}(\lambda)u\|_\infty&\leq & 2d3^d\|\Gamma\|_\infty\max_{\mathbf{i}\in J^d_{\ov{n}},\, h\in\{1,\ldots,d\}}\|\gamma_h\partial_{x_h}(\Phi_{\ov{n}}^\mathbf{i})\|_\infty\|\partial_{x_h}(R^\mathbf{i}_{\ov{n}}(\lambda)(\Phi_{\ov{n}}^\mathbf{i}u))\|_\infty \nonumber\\
& \leq  & 2d3^dK_2\frac{d_2}{\sqrt{|\lambda|}}\|u\|_\infty,
\end{eqnarray}
where $K_1=\max_{\mathbf{i}\in J^d_{\ov{n}}}\|\cL(\Phi_{\ov{n}}^\mathbf{i})\|_\infty$ and  $K_2=\|\Gamma\|_\infty \max_{\mathbf{i}\in J^d_{\ov{n}},\, h\in\{1,\ldots,d\}}\|\gamma_h\partial_{x_h}(\Phi_{\ov{n}}^\mathbf{i})\|_\infty$. We note that the constants $d_1$, $d_2$ and $R$ depend only on $B$ (with $B$ any fixed positive real number greater or equal than $\max_{i=1}^d\|b_i\|_\infty$) and on $\gamma=(\gamma_1,\ldots,\gamma_d)$ and that   the constants $K_1$ and $K_2$ depend only on $B$, $\Gamma$, $\gamma$  and on the functions $\{\Phi^{\mathbf{i}}_{\ov{n}}\}_{\mathbf{i}\in J^d_n}$. Now, by  \eqref{eq:c2} and \eqref{eq:c3} we can choose $R'>\max\{\frac{R}{\Gamma_0}, R\}$ large enough to get $\max\{\|C_2^{\ov{n}}(\lambda)\|,\|C_3^{\ov{n}}(\lambda)\|\}<1/4$ for all $\lambda\in \C$ with ${\rm Re}\lambda>R'$. So,  $\|C_1^{\ov{n}}(\lambda)+C_2^{\ov{n}}(\lambda)+C_3^{\ov{n}}(\lambda)\|<1/2$ via \eqref{eq:c1} for all $\lambda\in \C$ with ${\rm Re}\lambda>R'$ (we note that $R'$ depends on $B$, $\Gamma$, $\gamma$  and on the functions $\{\Phi^{\mathbf{i}}_{\ov{n}}\}_{\mathbf{i}\in J^d_n}$). This inequality, combined r with the equality \eqref{e:c0}, implies that the operator $C(\lambda)=(\lambda-\cL)S_{\ov{n}}(\lambda)$ is invertible in $\cL(C(Q^d))$ with  $\|(C(\lambda))^{-1}\|\leq 2$ for every $\lambda\in \C$ with ${\rm Re}\lambda>R'$. Since the operator $(\cL,D(\cL))$ generates a contractive $C_0$--semigroup in $C(Q^d)$ and so the operator $(\lambda-\cL)$ is injective for all $\lambda\in \C$ with ${\rm Re}\lambda>0$, it follows  that  $R(\lambda, \cL)=S_{\ov{n}}(\lambda)(C(\lambda))^{-1}$ and that by \eqref{e:ga3}
\begin{equation}\label{eq:res1}
\|R(\lambda, \cL)\|=\|S_{\ov{n}}(\lambda)(C(\lambda))^{-1}\|\leq \frac{2d_13^d}{\Gamma_0}\frac{1}{|\lambda|}
\end{equation}
for every $\lambda\in \C$ with ${\rm Re}\lambda>R'$. This inequality ensures that the operator $(\cL, D(\cL))$ is sectorial (see  \cite[Proposition 2.1.11]{L}), i.e., generates an analytic $C_0$--semigroup in $C(Q^d)$. Moreover, the identity $R(\lambda, \cL)=S_{\ov{n}}(\lambda)(C(\lambda))^{-1}$ implies that the operator $R(\lambda, \cL)$ is compact as $S_{\ov{n}}(\lambda)$ is compact, being $S_{\ov{n}}(\lambda)$ a sum of compact operators (observe  that $R^{\mathbf{i}}_{\ov{n}}=(\Gamma_{\ov{n}}^\mathbf{i})^{-1}R\left(\frac{\lambda}{\Gamma_{\ov{n}}^\mathbf{i}},\cL_1\right)$  and that $(\cL_1, D(\cL_1))$ generates an analytic  compact $C_0$--semigroup in $C(Q^d)$). So, the semigroup $(T(t))_{t\geq 0}$ is also compact, being it analytic and so norm continuous with compact resolvents. 

We now prove  that the estimates  \eqref{e.varcoeff} and \eqref{e.varcoeff-1} in Proposition \ref{p.ddim-gradiente} are shared by the operator $\cL$.

By  \eqref{eq:res1} the first part of Proposition \ref{p.ddim-gradiente}(2) is already proved.

Let $\lambda\in \C$ with ${\rm Re}\lambda>R'$ and $u\in D(\cL)=D(\cL_1)$. Then there exists $v\in C(Q^d)$ such that $R(\lambda, \cL_1)v=u$. So, by Proposition \ref{p.ddim-gradiente}(2), for every $ i=1,\ldots,d$, we have
\begin{equation}\label{eq:res2}
\|\sqrt{x_i}\partial_{x_i}u\|_\infty=\|\sqrt{x_i}\partial_{x_i}(R(\lambda,\cL_1)v)\|_\infty\leq \frac{d_2}{\sqrt{|\lambda|}}\|u\|_\infty=\frac{d_2}{\sqrt{|\lambda|}}\|\lambda u-\cL_1u\|_\infty.
\end{equation}
On the other hand, there exists also $w\in C(Q^d)$ such that $R(\lambda, \cL)w=u$. Then $\lambda u-\cL_1u=\left(1-\frac{1}{\Gamma}\right)\lambda u+\frac{1}{\Gamma}(\lambda u-\cL u)$ and hence, by \eqref{eq:res1} we have
\begin{eqnarray}\label{eq:res3}
\|\lambda u-\cL_1u\|_\infty &\leq &  \left(\frac{1}{\Gamma_0}+1\right)|\lambda|\|u\|_\infty+\frac{1}{\Gamma_0}\|\lambda u-\cL u\|_\infty\nonumber\\
&= & \left(\frac{1}{\Gamma_0}+1\right)|\lambda|\|R(\lambda, \cL)w\|_\infty+\frac{1}{\Gamma_0}\|w\|_\infty\nonumber\\
&\leq & \frac{2d_13^d}{\Gamma_0}\left(\frac{1}{\Gamma_0}+1\right)\|w\|_\infty+\frac{1}{\Gamma_0}\|w\|_\infty\nonumber\\
&\leq & K_3\left(\frac{1}{\Gamma_0}+1\right)\|w\|_\infty,
\end{eqnarray}
with $K_3=\max\left\{1, \frac{2d_13^d}{\Gamma_0}\right\}$. Combining \eqref{eq:res2} with \eqref{eq:res3} we obtain, for every 
$ i=1,\ldots,d$, that
\[
\|\sqrt{x_i}\partial_{x_i}(R(\lambda, \cL)w)\|_\infty\leq K_3\left(\frac{1}{\Gamma_0}+1\right)\frac{d_2}{\sqrt{|\lambda|}}\|w\|_\infty.
\]
Since $u$ is arbitrary and $R(\lambda, \cL)\colon C(Q^d)\to D(\cL)$ is bijective (for $\lambda\in \C$ with ${\rm Re}\lambda>R'$), the inequality \eqref{e.varcoeff-1} in Proposition \ref{p.ddim-gradiente}(2) is satisfied. Moreover,  the equality $R(\lambda, \cL)=S_{\ov{n}}(\lambda)(C(\lambda))^{-1}$ implies, for every $u\in C(Q^d)$ and $i=1,\ldots,d$, that $\sqrt{x_i}\partial_{x_i}(R(\lambda,\cL)u)\in C(Q^d)$  and that $\lim_{x_i\to 0^+}\sup_{x_i\in [0,M],\, j\in\{1,\ldots,d\}\setminus\{i\}}\sqrt{x_i}\partial_{x_i}(R(\lambda,\cL)u)=0$ via \eqref{limres-varcoeff} . So,  Proposition \ref{p.ddim-gradiente}(2) is valid.

One can prove that the estimates in  Proposition \ref{p.ddim-gradiente}(1),(2) hold for $\cL$   by arguing as in the proof of Proposition \ref{p.ddim-gradiente}.
\end{proof}


We now consider the following second order elliptic differential operator
\begin{equation}\label{e.opergen-dquad}
U=\Gamma(x)\sum_{i=1}^d[\gamma_i(x_i)x_i(1-x_i)\partial^2_{x_i}+b_i(x)\partial_{x_i}],\quad x\in Q^d,
\end{equation}
where   $\Gamma$, $b_i$ and $\gamma_i$,  for $i=1,\ldots,d$,  are continuous functions on $Q^d=[0,1]^d$ and on $[0,1]$ respectively. We assume that 

\begin{hypotheses}\label{hypon}The coefficients $\Gamma$, $b_i$ and $\gamma_i$, for $i=1,\ldots,d$, are continuous functions satisfying the following conditions.
\begin{itemize}
\item[\rm (i)] The functions  $\Gamma$ and $\gamma_i$, for $i=1,\ldots,d$, are strictly positive on $Q^d$ and on $[0,1]$ respectively. 
\item[\rm (ii)]  Let $b(x)=(b_1(x),\ldots, b_d(x))$ for $x\in Q^d$. Then $\langle b(x),\nu(x)\rangle \geq 0$ for every $x\in \partial Q^d$, where
 $\nu$ denotes the unit inward normal at $\partial Q^d$.
\item[\rm (iii)] There exist $\delta>0$ and $C>0$ such that, for every $i=1,\ldots,d$ and $x,\ x'\in Q^d$, we have
\begin{equation}\label{e.ass-2q}
|b_i(x)-b_i(x')|\leq C\sqrt{x_i},
\end{equation}
if  $|x_i|<\delta$ and  $x_i'=0$; while
\begin{equation}\label{e.ass-2qq}
|b_i(x)-b_i(x')|\leq  C\sqrt{1-x_i},
\end{equation}
if   $|1-x_i|<\delta$ and $x_i'=1$.
\end{itemize}
\end{hypotheses}

Proceeding in a similar way as in the proofs of Lemma \ref{l.posmin} and of Corollary \ref{c.dissoper} (or see \cite{CC1}), one shows that Hypotheses \ref{hypon}(i)--(ii) imply that a \textit{minimum principle} holds for the operator $(U, C^2(Q^d))$. Since $U1=0$, it follows that the operator $(U,C^2(Q^d))$  is dissipative and hence, $(U,C^2(Q^d))$ is closable in $C(Q^d)$ with closure $(\cU, D(\cU))$ a dissipative operator in $C(Q^d)$.
Moreover, we  have

\begin{theorem}\label{t.main-2} Under Hypotheses \ref{hypon}, the operator $(\cU, D(\cU))$ generates an analytic compact $C_0$--semigroup $(T(t))_{t\geq 0}$ of positive contractions in $C(Q^d)$. 
\end{theorem} 

For easy reading the proof of Theorem \ref{t.main-2} it is useful to introduce some notation and point out some results. 

In the sequel we follow the notation of Lemma \ref{l.part}. Let $n=2$ and set $\mathbf{I}^\mathbf{i}=\mathbf{I}^\mathbf{i}_2$ and $\Phi^\mathbf{i}=\Phi^\mathbf{i}_2$ for $\mathbf{i}\in J^d_2=\{1,2\}^d$. For a fixed $\mathbf{i}\in J^d_2$, the set  $\mathbf{I}^\mathbf{i}$ contains a unique vertex of $Q^d$, i.e., the vertex $V^\mathbf{i}$ of $Q^d$ with $(V^\mathbf{i})_h=0$ if $i_h=1$ and  $(V^\mathbf{i})_h=1$ if $i_h=2$. 
  If we denote by  $\psi_\mathbf{i}\colon Q^d\to Q^d$ the map given by setting
$\psi_\mathbf{i}(x)=y$,  for  $x\in Q^d$, with 
 $y_h=x_h$ if $i_h=0$ and $y_h=1-x_h$ if $i_h=2$. Clearly, $\psi_\mathbf{i}$ is a $ C^\infty$--diffeomorphism such that $\psi_\mathbf{i}(\mathbf{I}^\mathbf{i})=\mathbf{I}^\mathbf{i_0}$, where $\mathbf{i_0}$ denotes the element of $J^d_2$ with coordinates all equal to $1$.   Moreover,  the operator $\Psi_\mathbf{i}\colon C(Q^d)\to C(Q^d)$ defined by $\Psi_\mathbf{i}(u)=u\circ \psi_\mathbf{i}$ is a surjective isometry such that  $\Psi_\mathbf{i}(C^k(Q^d))=C^k(Q^d)$ for all $k\in\N$ (also $\Psi_\mathbf{i}(C^k(\mathbf{I}^\mathbf{i}))=C^k(\mathbf{I}^\mathbf{i_0})$). In particular,  $\Psi_\mathbf{i}$ transforms the operator $U_\mathbf{i}=U|_{\mathbf{I}^\mathbf{i}}$ into the  operator $L_\mathbf{i}$ of type \eqref{e.opergen-d} acting on the space $C^2_\diamond(\mathbf{I}^\mathbf{i_0})$. Indeed, we have, for every $u\in C^2_\diamond(\mathbf{I}^\mathbf{i_0})$, that 
\[
(U_\mathbf{i}\circ \Psi_\mathbf{i})(u)=\Gamma(x)\sum_{h=1}^d[\gamma_h(x_h)x_h(1-x_h)\partial_{y_h}^2u(\psi_\mathbf{i}(x))+b_h(x)c_h\partial_{y_h}u(\psi_\mathbf{i}(x))],
\] 
where $c_h=1$ if $i_h=1$ and $c_h=-1$  if $i_h=2$, and hence
\begin{eqnarray*}
& & (\Psi_\mathbf{i}^{-1}\circ U_\mathbf{i}\circ \Psi_\mathbf{i})(u)\\
& & \quad =\Gamma(\psi_\mathbf{i}^{-1}(y))\sum_{h=1}^d[\gamma_h((\psi_\mathbf{i}^{-1}(y))_h)y_h(1-y_h)\partial_{y_h}^2u(y)+b_h(\psi_\mathbf{i}^{-1}(y))c_h\partial_{y_h}u(y)].
\end{eqnarray*}
Now, we observe that if we set $\tilde{\gamma}_h(y)=\gamma_h((\psi_\mathbf{i}^{-1}(y))_h)(1-y_h)$ and $\tilde{b}_h(y)=b_h(\psi_\mathbf{i}^{-1}(y))c_h$  for $y\in I^\mathbf{i_0}$ and $h=1,\ldots,d$, then the functions $\tilde{\gamma}_h,\, \tilde{b}_h$ are continuous on $I^\mathbf{i_0}$ and on $[0,1/2]$ respectively, and each  function $\tilde{\gamma}_h$ is strictly positive (as $1/3\leq 1-y_h\leq 1$ for every $h$). Also, by Hypothesis \ref{hypon}(iii) we have, for every $h=1,\ldots,d$ and $y,\, y'\in \mathbf{I}^\mathbf{i_0} $ with $y'_h=0$ and $|y_h|\leq \delta$, that 
\[
|\tilde{b}_h(y)-\tilde{b}_h(y')|\leq C|y_h|.
\]
Finally, Hypothesis \ref{hypon}(ii) implies that if we set  $\tilde{b}=(\tilde{b}_1, \ldots, \tilde{b}_d)$ then $\langle \tilde{b}, \nu\rangle \geq 0$ on $\partial\mathbf{I}^\mathbf{i_0}_0$.  

Since Hypotheses \ref{hypo} are fulfilled,  we can apply Theorem \ref{t.main1}  to conclude that the closure $(\cL_\mathbf{i}, D(\cL_\mathbf{i}) )$ of the operator  $(L_\mathbf{i}, C^2_\diamond(\mathbf{I}^\mathbf{i_0}))$ generates an  analytic compact  $C_0$--semigroup in $C(\mathbf{I}^\mathbf{i_0})$ of positive contractions. Moroever,  the operator  $(\cL_\mathbf{i}, D(\cL_\mathbf{i}) )$ satisfies  Proposition \ref{p.ddim-gradiente}. So, by similarity the closure $(\cU_\mathbf{i}, D(\cU_\mathbf{i}))$ of the operator $(U_\mathbf{i}, C^2_{\diamond'}(\mathbf{I}^\mathbf{i}))$ (here, $u\in C^2_{\diamond'}(\mathbf{I}^\mathbf{i})$ if $u\in C^2(\mathbf{I}^\mathbf{i})$ and $\partial_{x_h}u(x)=0$ if either  $x_h=2/3$ and $i_h=1$ or $x_h=1/3$ and $i_h=2$)  generates an analytic compact  $C_0$--semigroup in $C(\mathbf{I}^\mathbf{i})$ of positive contractions and  satisfies the proper analogue of Proposition \ref{p.ddim-gradiente}. In particular, the proper analogue of  Proposition \ref{p.ddim-gradiente}(2)  turns out to be
\begin{itemize}
\item[\rm (2)'] There exist  $d_1, d_2, \, R>0$  depending only on $B$ (with $B$ any positive real number $\geq \max_{i=1}^d\|b_i\|_\infty$) and on $\Gamma$, $\gamma=(\gamma_1,\ldots, \gamma_d)$ such that, for every $\lambda\in\C$ with ${\rm Re}\lambda>R$, $h=1,\ldots, d$, $\mathbf{i}\in J^d_2$ and $u\in C(\mathbf{I}^\mathbf{i})$, we have
\[
\|R(\lambda,\cU_\mathbf{i}) \|\leq \frac{d_1}{|\lambda|},
\]
\[
\|\sqrt{x_h (1-x_h)}\partial_{x_h}(R(\lambda,\cU_\mathbf{i}) u)\|_\infty\leq \frac{d_2}{\sqrt{|\lambda|}}\|u\|_\infty.
\]
\end{itemize}

We are now able to show Theorem \ref{t.main-2}.

\begin{proof} By (2)' above we have,  for every $\mathbf{i}\in J^d_2$ and $\lambda\in\C$ with ${\rm Re}\lambda>R$, that 
\begin{equation}\label{e.diseq-1}
\|R(\lambda, \cU_\mathbf{i})\|\leq \frac{d_1}{|\lambda|}.
\end{equation}
So,  for every $\lambda\C$ with ${\rm Re}\lambda>R$ we can consider the operator $S(\lambda) \colon C(Q^d)\to C(Q^d)$ defined by
\begin{equation}\label{e.ident-11}
S(\lambda)u=\sum_{\mathbf{i}\in J^d_2}\Phi^\mathbf{i}R(\lambda, \cU_\mathbf{i})(\Phi^\mathbf{i}u),\quad u\in C(Q^d).
\end{equation}
Hence, by \eqref{e.diseq-1}  we have, for every $\lambda\in \C$ with ${\rm Re}\lambda>R$, that 
\begin{equation}\label{e.diseq-2}
\|S(\lambda)\|\leq \frac{2^dd_1}{|\lambda|}.
\end{equation} 
We observe that the previous considerations on the differential operators $\cU_\mathbf{i}$ ensure, for every  $\mathbf{i}\in J^d_2$ and $u\in C(Q^d)$,  that
\begin{equation}\label{e.ident-21}
\cU(\Phi^\mathbf{i}R(\lambda, \cU_\mathbf{i})(\Phi^\mathbf{i}u))=\cU_\mathbf{i}(\Phi^\mathbf{i}R(\lambda, \cU_\mathbf{i})(\Phi^\mathbf{i}u))
\end{equation}  
and, for every $u,\, v\in D(\cU_\mathbf{i})$, that
\begin{equation}\label{e.ident-3}
\cU_\mathbf{i}(uv)=u\cU_\mathbf{i}(v)+v\cU_\mathbf{i}(u)+\Gamma(x)\sum_{h=1}^d\gamma_h(x_h)x_h(1-x_h)\partial_{x_h}u\partial_{x_h}v.
\end{equation}
By \eqref{e.ident-21} and \eqref{e.ident-3} we obtain, for every $\lambda\in\C$ with ${\rm Re}\lambda>R$ and $u\in C(Q^d)$, that
\begin{eqnarray}\label{e.diseq-3}
(\lambda-\cU)S(\lambda)(u)&=& \lambda S(\lambda) (u)-\sum_{\mathbf{i}\in J^d_2}\cU(\Phi^\mathbf{i}R(\lambda, \cU_\mathbf{i})(\Phi^\mathbf{i}u)) \nonumber\\
&=&\lambda S(\lambda) (u)-\sum_{\mathbf{i}\in J^d_2}\cU_\mathbf{i}(\Phi^\mathbf{i}R(\lambda, \cU_\mathbf{i})(\Phi^\mathbf{i}u))  \nonumber\\
&=& \sum_{\mathbf{i}\in J^d_2}\Phi^\mathbf{i}(\lambda-\cU_\mathbf{i})R(\lambda, \cU_\mathbf{i})(\Phi^\mathbf{i}u)-\sum_{\mathbf{i}\in J^d_2}\cU_\mathbf{i}(\Phi^\mathbf{i})R(\lambda, \cU_\mathbf{i})(\Phi^\mathbf{i}u)  \nonumber\\
& -&\Gamma(x)\sum_{\mathbf{i}\in J^d_2}\sum_{h=1}^d\gamma_h(x_h)x_h(1-x_h)\partial_{x_h}(R(\lambda, \cU_\mathbf{i})(\Phi^\mathbf{i}u)) \partial_{x_h}\Phi^\mathbf{i}  \nonumber\\
&=:&(I+B(\lambda)+C(\lambda))(u).
\end{eqnarray}
By \eqref{e.diseq-1} we deduce, for every $\lambda\in\C$ with ${\rm Re}\lambda>R$ and $u\in C(Q^d)$, that
\begin{equation}\label{e.diseq-4}
\|B(\lambda)u\|_\infty\leq\frac{d_3}{|\lambda|}\|u\|_\infty,
\end{equation}
where $d_3=2^d d_1 \max_{\mathbf{i}\in J^d_2}\|\cU_\mathbf{i}(\Phi^\mathbf{i})\|_\infty$. 

Applying again the property (2)' above we obtain, for every $\lambda\in\C$ with ${\rm Re}\lambda>R$ and $u\in C(Q^d)$, that
\begin{eqnarray}\label{e.diseq-5}
\|C(\lambda)u\|_\infty&\leq & d_4 \sup_{h=1,\ldots,d,\ \mathbf{i}\in J^d_2}\|x_h(1-x_h)\partial_{x_h}(R(\lambda, \cU_\mathbf{i})(\Phi^\mathbf{i}u)) \|_\infty\nonumber\\
&\leq & d_4\frac{d_2}{\sqrt{|\lambda|}}\sup_{h=1,\ldots,d,\ \mathbf{i}\in J^d_2}\|\Phi^\mathbf{i}u \|_\infty \nonumber\\
&\leq & d_4\frac{d_2}{\sqrt{|\lambda|}}\|u \|_\infty,
\end{eqnarray}
where $d_4=2^ddM$ with $M=\|\Gamma\|_\infty\max_{\mathbf{i}\in J^d_2,\, h=1,\ldots,d}\|\gamma_h\partial_{x_h}\Phi^{\mathbf{i}}\|_\infty$ (hence, $M$ depends only on $\Gamma$, $\gamma$ and on the functions $\{\Phi^\mathbf{i}\}_{\mathbf{i}\in J^d_2}$.

By \eqref{e.diseq-4} and \eqref{e.diseq-5}  we can choose  $R'\geq R$ such that $\|B(\lambda)+C(\lambda)\|<1/2$ for all $\lambda\in \C$ with ${\rm Re}\lambda>R'$ and hence, the operator $D(\lambda)=(\lambda -\cU)S(\lambda)$ is invertible in $\cL(C(Q^d))$ with $\|((D(\lambda))^{-1}\|\leq 2$. So, there exists $R(\lambda, \cU)=S(\lambda)(D(\lambda))^{-1}$ and  and satisfies by\eqref{e.diseq-2} 
\begin{equation}\label{e.diseq-9}
\|R(\lambda, \cU)\|\leq \frac{2^{d+1}d_1 }{|\lambda|}
\end{equation}
whenever $\lambda- \cU$ is injective, in particular, for $\lambda>0$ as the operator $\cU$ is dissipative (observe that $R(\lambda, \cU)$ is also compact as $S(\lambda)$ is compact). Since $(\cU, D(\cU))$ is also densely defined, by Lumer-Phillips theorem this fact ensures that the operator  $(\cU, D(\cU))$ generates a $C_0$--semigroup $(T(t))_{t\geq 0}$ of contractions in $C(Q^d)$. So, for every $\lambda\in\C$ with ${\rm Re}\lambda>R'$ we have that $R(\lambda, \cU)=S(\lambda)(D(\lambda))^{-1}$  and satisfies inequality \eqref{e.diseq-9}.

Finally,  from  \eqref{e.diseq-9} it follows that the operator $(\cU, D(\cU))$ is sectorial (see  \cite[Proposition 2.1.11]{L}), i.e., generates an analytic $C_0$--semigroup in $C(Q^d)$. 
Since the semigroup is analytic, hence norm--continuous, and the differential operator $(\cU, D(\cU))$ has compact resolvent, the semigroup is also compact.
\end{proof}

We end this section with the following result which could be useful for further developments.

\begin{prop}\label{p.dimd-gradvariablecompleto} Let $B\geq \max_{i=1}^d\|b_i\|_\infty$. Under Hypotheses \ref{hypon}, the operator $(\cU, D(\cU))$  satisfies the following properties.
\begin{itemize}
\item[\rm (1)] There exist $K,\alpha, \ov{t}>0$ depending on $B$ and on $\Gamma$, $\gamma$, such that, for every $u\in C(Q^d)$ and $i=1,\ldots, d$, we have
\[
\|t\cU T(t)\|\leq Ke^{\alpha t}, \quad t\geq 0,
\]
\[
\|\sqrt{x_i(1-x_i)} \partial_{x_i}(T(t) u)\|_\infty\leq \frac{Ke^{\alpha t}}{\sqrt{t}}\|u\|_\infty,\quad 0<t<\overline{t},
\]
 \[
 \|\sqrt{x_i(1-x_i)} \partial_{x_i}(T(t) u)\|_\infty\leq Ke^{\alpha t}\|u\|_\infty, \quad  t\geq \overline{t}.
 \]
Moreover, for every $i=1,\ldots, d$ and $u\in C(Q^d)$,  $\sqrt{x_i(1-x_i)} \partial_{x_i}(T(t) u)\in C(Q^d)$ and 
\[
\lim_{x_i\to 0^+, 1^-}\sup_{x_j\in [0,1],\, j\in \{1,\ldots,d\}\setminus\{i\}}\sqrt{x_i(1-x_i)} \partial_{x_i}(T(t) u)=0.
\]
\item[\rm (2)] There exist  $d_1,d_2, R>0$  depending on $B$ and on $\Gamma$, $\gamma$, such that, for every $\lambda\in\C$ with ${\rm Re}\lambda>R$,  $u\in C(Q^d)$ and $i=1,\ldots, d$, we have
\[
\|R(\lambda,\cU) u\|_\infty\leq \frac{d_1}{|\lambda|}\|u\|_\infty,
\]
\[
\|\sqrt{x_i(1-x_i)} \partial_{x_i}(R(\lambda,\cU) u)\|_\infty\leq \frac{d_2}{\sqrt{|\lambda|}}\|u\|_\infty.
\]
Moreover, for every $i=1,\ldots, d$ and $u\in C(Q^d)$,  $\sqrt{x_i(1-x_i)} \partial_{x_i}(R(\lambda,\cU))\in C(Q^d)$ and 
\[
\lim_{x_i\to 0^+,1^-}\sup_{x_j\in [0,1],\, j\in \{1,\ldots,d\}\setminus\{i\}}\sqrt{x_i(1-x_i)} \partial_{x_i}(R(\lambda,\cU) u)=0.
\]
\item[\rm (3)] There exist $\overline{\varepsilon}>0$, $C>0$ and $D>0$ depending on $B$ and on $\Gamma$, $\gamma$, such that,  for every $0<\varepsilon<\overline{\varepsilon}$,  $i=1,\ldots, d$  and    
$u\in D(\cU)$, we have 
\[ 
\|\sqrt{x_i(1-x_i) }\partial_{x_i}u\|_\infty\leq \frac{C} {\varepsilon}  \|u\|_\infty+ D\varepsilon \|\cU u\|_\infty.
 \]
\end{itemize}
\end{prop}

\begin{proof}  The result follows argumenting as in the end of the proof of Theorem \ref{t.main1}.\end{proof}

\begin{example}\rm In the following examples assume that the coefficients $\Gamma$  and $\gamma_i$, for $i=1,\ldots,d$, are strictly positive continuous functions  on $Q^d$ and on $[0,1]$ respectively. 

(1)   Let $\{c_i\}_{i=1}^d\subset C([0,1])$ such that $c_i(0)=c_i(1)=0$ and there exist $0<\delta<1$ and $C>0$ such that
\[
|c_i(x)|\leq C		\sqrt{x},\quad \mbox{ if } 0\leq x\leq \delta, \mbox{ and }   |c_i(x)|\leq C\sqrt{1-x},\quad \mbox{ if } 1-\delta\leq x\leq 1,
\]
for every $i=1,\ldots, d$. Let $\{m_i\}_{i=1}^d\subset C(Q^{d-1})$. 

Now, for each $i=1,\ldots, d$ and $x\in Q^d$ denote $x^i=(x_1,\ldots, x_{i-1}, x_{i+1}, \ldots, x_d)$ and set $b_i(x)=c_i(x_i)m_i(x^i)$.  Then $\{b_i\}_{i=1}^d\subset C(Q^d)$. Moreover, we have, for every $x,\, x'\in Q^d$ with $|x_i|\leq \delta$ and $x'_i=0$ and $i=1,\ldots, d$,  that
\[
|b_i(x)-b_i(x')|=|c_i(x_i)m_i(x^i)|\leq C'\sqrt{x_i},
\]
with $C'=C\max_{i=1}^d\|m_i\|_\infty$.
On the other hand, we  have, for every $x,\, x'\in Q^d$ with $|1-x_i|\leq \delta$ and $x'_i=1$ and $i=1,\ldots, d$,  that
\[
|b_i(x)-b_i(x')|=|c_i(x_i)m_i(x^i)|\leq C'\sqrt{1-x_i}.
\]
Finally, it is easy to verify that $\langle b,\nu\rangle\geq 0$ on $\partial Q^d$. So, if we consider the second order differential operator
\[
U=\Gamma(x)\sum_{i=1}^d[\gamma_i(x_i)x_i(1-x_i)\partial^2_{x_i}+b_i(x)\partial_{x_i}],
\]
with $\Gamma$ and $\gamma_i$, for $i=1,\ldots,d$, strictly positive continuous functions on $Q^d$ and on $[0,1]$ respectively, then
the closure $(\cU, D(\cU))$ of operator $(U, C^2(Q^d))$ generates an analytic compact $C_0$--semigroup in $C(Q^d)$ of positive contractions.

(2) Let $\{c_i\}_{i=1}^d\subset C([0,1])$ such that $c_i\geq 0$ on $[0,1]$ and  there exist $0<\delta<1$ and $C>0$ such that
\[
|c_i(x)-c(0)|\leq C\sqrt{x},\quad \mbox{ if } 0\leq x\leq \delta, \mbox{ and }   |c_i(x)-c_i(1)|\leq C\sqrt{1-x},\quad \mbox{ if } 1-\delta\leq x\leq 1,
\]
for every $i=1,\ldots, d$. Next, for each $i=1,\ldots, d$ and $x\in Q^d$ set $\tilde{c}(x)=\sum_{i=1}^dc_i(x_i)$ and  $b_i(x)=c_i(x_i)-\tilde{c}(x)x_i(1-x_i)$. These type of coefficients was considered in \cite{S1,S2,CC}. Then we have, for every $x,\, x'\in Q^d$ with $|x_i|\leq \delta$ and $x'_i=0$ and $i=1,\ldots, d$,  that
\begin{eqnarray*}
|b_i(x)-b_i(x')|&=&|c_i(x_i)-\tilde{c}(x)x_i(1-x_i)-c_i(0)|\\
&\leq & |c_i(x_i)-c_i(0)|+|\tilde{c}(x)(1-x_i)||x_i|\leq C'\sqrt{x_i},
\end{eqnarray*}
with $C'=C+\|\tilde{c}\|_\infty$. 

On the other hand, we  have, for every $x,\, x'\in Q^d$ with $|1-x_i|\leq \delta$ and $x'_i=1$ and $i=1,\ldots, d$,  that
\begin{eqnarray*}
|b_i(x)-b_i(x')|&=&|c_i(x_i)-\tilde{c}(x)x_i(1-x_i)-c_i(1)|\\
&\leq &|c_i(x_i)-c_i(1)|+|\tilde{c}(x)x_i||1-x_i|\leq C'\sqrt{1-x_i}.
\end{eqnarray*}
Finally, it is easy to verify that $\langle b,\nu\rangle\geq 0$ on $\partial Q^d$. So, if we consider the second order differential operator
\[
U=\Gamma(x)\sum_{i=1}^d[\gamma_i(x_i)x_i(1-x_i)\partial^2_{x_i}+b_i(x)\partial_{x_i}],
\]
with $\Gamma$ and $\gamma_i$, for $i=1,\ldots,d$, strictly positive continuous functions on $Q^d$ and on $[0,1]$ respectively, then
the closure $(\cU, D(\cU))$ of operator $(U, C^2(Q^d))$ generates an analytic  compact $C_0$--semigroup in $C(Q^d)$ of positive contractions.
\end{example}


 \bigskip
\bibliographystyle{plain}

\end{document}